\documentclass{article}

\usepackage{amsmath,amssymb,amsfonts,amsthm,amsbsy,mathtools}
\usepackage{graphicx,colortbl}
\usepackage[dvipsnames]{xcolor}
\usepackage{hyperref}
\usepackage{enumerate}
\hypersetup{
    colorlinks=true,
    linkcolor=blue,
    citecolor=darkgray
}
\usepackage[T1]{fontenc}
\usepackage{footnote,pifont,dsfont}
\usepackage[mathscr]{eucal}
\usepackage{fullpage}
\usepackage{thmtools}
\usepackage{dirtytalk,tikz}
\usetikzlibrary{decorations.pathreplacing, arrows,positioning}
\usepackage{caption}
\usepackage{verbatim}

\newtheorem{theorem}{Theorem}
\newtheorem{corollary}[theorem]{Corollary}
\newtheorem{proposition}[theorem]{Proposition}

\newtheorem{lemma}[theorem]{Lemma}

\theoremstyle{definition}
\newtheorem{defn}[theorem]{Definition}
\newtheorem{conjecture}[theorem]{Conjecture}

\theoremstyle{remark}
\newtheorem{rem}[theorem]{Remark}

\newtheorem{examplex}[theorem]{Example}
\newenvironment{example}
  {\pushQED{\qed}\examplex}
  {\popQED\endexamplex}

\definecolor{LightGrey}{RGB}{220, 220, 220}

\newcommand{\mA}{\ensuremath{\mathcal{A}}}

\newcommand{\mD}{\ensuremath{\mathcal{D}}}

\newcommand{\CC}{\ensuremath{\mathcal{C}}}
\newcommand{\mF}{\ensuremath{\mathcal{F}}}

\newcommand{\mH}{\ensuremath{\mathcal{H}}}
\newcommand{\mR}{\ensuremath{\mathcal{R}}}
\newcommand{\mS}{\ensuremath{\mathcal{S}}}
\newcommand{\mT}{\ensuremath{\mathcal{T}}}

\newcommand{\mV}{\ensuremath{\mathcal{V}}}
\newcommand{\sing}{\ensuremath{\mathcal{V}}}

\newcommand{\E}{\ensuremath{\mathbb{E}}}

\newcommand{\C}{\mathbb{C}}
\newcommand{\Z}{\mathbb{Z}}
\newcommand{\R}{\mathbb{R}}
\newcommand{\N}{\mathbb{N}}
\newcommand{\Q}{\mathbb{Q}}
\newcommand{\bP}{\mathbb{P}}

\newcommand{\bA}{\ensuremath{\mathbf{A}}}
\newcommand{\ba}{\ensuremath{\mathbf{a}}}

\newcommand{\bi}{\ensuremath{\mathbf{i}}}
\newcommand{\bj}{\ensuremath{\mathbf{j}}}
\newcommand{\bk}{\ensuremath{\mathbf{k}}}

\newcommand{\bm}{\ensuremath{\mathbf{m}}}
\newcommand{\bn}{\ensuremath{\mathbf{n}}}

\newcommand{\br}{\ensuremath{\mathbf{r}}}
\newcommand{\rr}{\ensuremath{\mathbf{r}}}

\newcommand{\uu}{\ensuremath{\mathbf{u}}}
\newcommand{\vv}{\ensuremath{\mathbf{v}}}
\newcommand{\ww}{\ensuremath{\mathbf{w}}}

\newcommand{\bw}{\ensuremath{\mathbf{w}}}
\newcommand{\bx}{\ensuremath{\mathbf{x}}}
\newcommand{\xx}{\ensuremath{\mathbf{x}}}
\newcommand{\bX}{\ensuremath{\mathbf{X}}}
\newcommand{\bY}{\ensuremath{\mathbf{Y}}}
\newcommand{\by}{\ensuremath{\mathbf{y}}}

\newcommand{\bz}{\ensuremath{\mathbf{z}}}
\newcommand{\zz}{\ensuremath{\mathbf{z}}}
\newcommand{\one}{\ensuremath{\mathbf{1}}}
\newcommand{\zero}{\ensuremath{\mathbf{0}}}

\newcommand{\grad}{\ensuremath{\nabla}}

\newcommand{\bss}{\mathbf{s}}

\newcommand{\htbm}{\hat{\bm}}
\newcommand{\bs}{\mathbf{s}}

\newcommand{\bmu}{\ensuremath{\boldsymbol \mu}}
\newcommand{\bSigma}{\ensuremath{\boldsymbol \Sigma}}

\newcommand{\btheta}{\ensuremath{\mathbf{\theta}}}
\newcommand{\bone}{\ensuremath{\mathbf{1}}}

\title{Central Limit Theorems via Analytic Combinatorics in Several Variables}

\author{Stephen Melczer and Tiadora Ruza}

\begin{document}
\maketitle

\begin{abstract}
The field of analytic combinatorics is dedicated to the creation of effective techniques to study the large-scale behaviour of combinatorial objects. Although classical results in analytic combinatorics are mainly concerned with univariate generating functions, over the last two decades a theory of analytic combinatorics in several variables (ACSV) has been developed to study the asymptotic behaviour of multivariate sequences. In this work we survey ACSV from a probabilistic perspective, illustrating how its most advanced methods provide efficient algorithms to derive limit theorems, and comparing the results to past work deriving limit theorems. Using the results of ACSV, we provide a SageMath package that can automatically compute (and rigorously verify) limit theorems for a large class of combinatorial generating functions. To illustrate the techniques involved, we also establish an explicit local central limit theorems for a family of combinatorial classes whose generating functions are linear in the variables tracking each parameter. Applications covered by this result include the distribution of cycles in certain restricted permutations (proving a limit theorem stated as a conjecture in recent work of Chung et al.~\cite{ChungDiaconisGraham2021}), integer compositions, and $n$-colour compositions with varying restrictions and values tracked. Key to establishing these explicit results in an arbitrary dimension is an interesting symbolic determinant, which we compute by conjecturing and then proving an appropriate $LU$-factorization. It is our hope that this work provides readers a blueprint to apply the powerful tools of ACSV in their own work, making them more accessible to combinatorialists, probabilists, and those in adjacent fields.
\end{abstract}

Let $(X_n)$ be a sequence of random variables. A \emph{limit theorem} (or \emph{limit law}) for $X_n$ is an approximation of the cumulative distribution functions $\bP(X_n \leq k)$ as $n\rightarrow\infty$. A \emph{local limit theorem} is an approximation of the exact probabilities $\bP(X_n = k)$ as $n\rightarrow\infty$. A \emph{central limit theorem (CLT)} or \emph{local central limit theorem (LCLT)} compares these probabilities to a \emph{normal density function}
\[ \phi_{\mu,\sigma}(x) = \frac{1}{\sqrt{2\pi}} e^{-(x-\mu)^2/2\sigma}. \]
The \emph{classical CLT} states that if $A_1,A_2,\dots$ is a sequence of independent identically distributed random variables with an expected value $\mu$ and a finite variance $\sigma^2>0$ then the sequence of random variables
\[ X_n = A_1 + \cdots + A_n \]
converges in distribution, after rescaling, to the standard normal distribution, so that
\[ \lim_{n \rightarrow \infty} \bP\left[\frac{X_n - n\mu}{\sigma\sqrt{n}} \leq x\right] = \int_{-\infty}^x \phi_{0,1}(t)dt \]
for all real $x$. Note that we rescale the sequence $X_n$ so that its limit is a fixed distribution instead of one that varies with $n$.

The classical CLT has a long history (see Section~\ref{sec:history} below for a brief recap), and has been generalized to weaken its assumptions, give explicit rates of convergence, and work in more abstract settings. In $d$-dimensions, a \emph{multivariate (L)CLT} compares probabilities to \emph{multivariate normal density functions}
\[ \phi_{\bmu,\bSigma}(\xx) = \frac{1}{(2\pi)^{d/2}} \exp\left(-\frac{1}{2}(\xx-\bmu)^T\bSigma^{-1}(\xx-\bmu)\right). \]
The \emph{classical multivariate CLT} states that if $(\bA_n)$ is a sequence of $d$-dimensional independent identically distributed random variables with an expected value vector $\bmu$ and a positive definite covariance matrix $\bSigma$ then the sequence of $d$-dimensional random variables
\[ \bX_n = \frac{\bA_1+\cdots+\bA_n-n\bmu}{\sqrt{n}} \]
converges in distribution to the multivariate normal distribution with density function $\phi_{0,\bSigma}(\xx)$. 

In combinatorial contexts, one often has a \emph{combinatorial class} $(\mathcal{C},|\cdot|)$ defined by a set of objects $\mathcal{C}$ and a \emph{size function} $|\cdot|:\mathcal{C}\rightarrow\N$ such that there are only a finite number of objects of any given size. A \emph{$d$-dimensional parameter} on such a class is any map $\chi:\mathcal{C}\rightarrow\N^d$, and the \emph{multivariate generating function} with respect to the triple $(\mathcal{C},|\cdot|,\chi)$ is the $(d+1)$-dimensional power series\footnote{Throughout this article we use the multi-index notation $\bz^\bi=z_1^{i_1}z_2^{i_2}\cdots z_d^{i_d}$ for any $d$-dimensional vectors $\bz$ and $\bi$.}
\[ C(\bz,t) = \sum_{n \geq 0}\left(\sum_{\bi \in \N^d} c_{\bi,n}\bz^\bi \right)t^n \]
where $c_{\bi,n}$ denotes the number of objects in $\mathcal{C}$ with size $n$ and parameter value $\bi$. Well-established combinatorial theories exist to derive algebraic, differential, or functional equations satisfied by the generating functions of many types of combinatorial classes (see, for instance,~\cite{FlajoletSedgewick2009,GouldenJackson1983,Stanley1999,Wilf2006}). Thus, one often has a multivariate generating function $C(\bz,t)$ encoded in some way, and wants to determine the limiting behaviour of the $d$-dimensional array $(c_{\bi,n})_{\bi\in\N^d}$ as a function of $\bi$ when $n\rightarrow\infty$. 

Strikingly, the theory of \emph{analytic combinatorics} shows how central limit theorems for many combinatorial parameters can be derived directly from a study of the singular behaviour of $C(\bz,t)$. Flajolet and Sedgewick~\cite[Chapter IX]{FlajoletSedgewick2009} gives a detailed introduction, and Hwang~\cite{Hwang2024} lists more than 25 limit laws proven by the late Philippe Flajolet using this framework, including patterns in words, cost analyses of various sorting algorithms, parameters in different kinds of trees, polynomials with restricted coefficients, and ball-urn models. 

The classical theory of analytic combinatorics is chiefly concerned with the derivation of asymptotics of univariate generating functions. In typical cases, one applies \emph{transfer theorems} to deduce the asymptotic behaviour of a generating function's coefficients from series expansions of the generating function near its singularities. For many classes of generating functions, it is essentially algorithmic (on examples occurring in practice) to compute the singularities determining asymptotics, apply transfer theorems, and find asymptotic behaviour. General multivariate results in analytic combinatorics can be traced back to work of Bender, Richmond, and collaborators in the 1980s (see Section~\ref{sec:history} below for more details). Although such results apply to many applications in combinatorics and the analysis of algorithms, they capture only one type of multivariate singularity and require one to verify certain properties of multivariate generating functions (such as grouping singularities by moduli), which can be computationally expensive.

These issues led to the development of \emph{analytic combinatorics in several variables (ACSV)}~\cite{Melczer2021,PemantleWilsonMelczer2024} starting in the early 2000s. Following the framework of univariate analytic combinatorics, the goal is to derive effective methods to take an encoding of a multivariate generating function and return asymptotics or limit theorems of its series coefficients. Most results in ACSV focus on sequences with multivariate rational generating functions: such sequences appear naturally in many applications, and many functions with more complicated singular behaviour can be encoded by rational functions in a higher number of variables (for instance, every algebraic function in $d$ variables can be encoded as an explicit subseries of a rational function in $2d$ variables~\cite{DenefLipshitz1987}). The crux of this paper is to illustrate both the power of ACSV for proving limit theorems and the (in many cases completely) automatic nature of the required computations. 

\subsubsection*{Paper Structure}

In addition to summarizing the existing literature on combinatorial limit theorems, and clarifying how the methods of ACSV build on and modify these results, we hope to illustrate how explicit and algorithmic modern methods for verifying central limit theorems are. To this end, we provide easy to use theorems\footnote{The authors of this paper were introduced to Conjecture~\ref{conj:LCLT} by an author of Chung et al.~\cite{ChungDiaconisGraham2021}, who lamented on the opacity of some past references using ACSV to establish limit theorems.}, employing the ACSV template which can be applied by researchers in their own works. In addition to illustrating these methods on a collection of applications, including resolving a conjecture stated in Chung et al.~\cite{ChungDiaconisGraham2021} on the behaviour of cycles in a class of restricted permutations\footnote{We also note that this work began concurrently with the publication of the textbook~\cite{Melczer2021} of the first author (which includes details on how to computationally prove the specific limit theorem in~\eqref{eq:LCLT} for fixed -- not arbitrary -- dimension).}, we provide a SageMath package that can rigorously compute and prove central limit theorems for a large variety of combinatorial classes tracking parameters with multivariate rational generating functions, available at

\begin{center}
\url{https://github.com/ACSVMath/Limit-Theorems-For-Combinatorial-Parameters} 
\end{center}

Section~\ref{sec:history} gives a brief history of central limit theorems, paying particular attention to the use of analytic methods. After surveying results in a probabilistic context, we describe the development of CLTs in combinatorics. Section~\ref{sec:ACSVandCLT} then gives an overview of the methods of ACSV, illustrating how multivariate LCLTs follow from the general theory. Section~\ref{sec:limitExs} illustrates how these results lead to a general limit theorem with explicit constants holding for a number of combinatorial objects, and Section~\ref{sec:GenForm} proves the explicit constants appearing in this limit theorem using a guess-and-check method to compute a symbolic determinant through $LU$-factorization. 

\section{Analytic Methods for Central Limit Theorems}
\label{sec:history}

The application of analytic methods to prove limit theorems in probability and combinatorics has a long and wide-ranging history. It is impossible to give a full account of such a vast topic in this space, so our presentation is a broad overview tailored to the context of our work. 

\subsection{Probabilistic CLTs}
The proto-history of the CLT goes back at least as far\footnote{The history of the CLT is covered in great detail by Fischer~\cite{Fischer2011}, from which we have adapted some of our historical details.} as a 1733 offprint \emph{Approximatio ad summam terminorum binomii $(a+b)^n$ in seriem expansi} of Abraham de Movire (printed in English in the 1738 edition of his seminal Doctrine of Chances~\cite{Moivre1738}). Motivated by the computation of explicit bounds for the Law of Large Numbers, de Moivre used Stirling's approximation for $n!$ (which de Moivre independently approximated around the same time as Stirling) to deduce\footnote{In historical formulas we have updated some notation; for instance, de Moivre referred to the exponential function only by its series expansion, while we use its symbolic form to align with modern presentations of the CLT.}
\[ \bP\left[\left| X_n - \frac{n}{2} \right| \leq t \right] \approx \frac{4}{\sqrt{2\pi n}} \int_0^t e^{-2x^2/n}dx \]
for large $n$, where $X_n = A_1 + \cdots + A_n$ and each $A_k$ is a random variable taking the value 0 with probability $1/2$ and the value 1 with probability $1/2$.

Perhaps the first systematic uses of analytic methods to derive CLT-like results are due to Laplace. Laplace's approach to the CLT built off of his ground-breaking work~\cite{Laplace1774} on the approximation of parameterized integrals in the 1770s. Still influential to this day, \emph{Laplace's method} (in its classical form) is used to approximate integrals of the form
\[ \int_a^b A(x)e^{-n\phi(x)}dx \]
where $A$ and $\phi$ are analytic functions with $\phi$ minimized over $[a,b]$ at a unique point $c \in (a,b)$ such that $\phi$ has a Taylor expansion at $x=c$ that begins with a quadratic term. In an 1810 memoir, Laplace~\cite{Laplace1810} introduced the concept of \emph{characteristic functions} by making a change of variable $t=e^{ix}$ in integral representations for the sum of certain independent and identically distributed random variables. He then (formally) used his method to asymptotically approximate the probability that this sum lies between two factors at $\sqrt{n}$-scale by integrating a normal density. Although not working with a modern standard of rigor, Laplace's techniques continue to be influential to this day.

Many famous mathematicians in the nineteenth century worked on topics related to central limit theorems, including Gauss~\cite{Gauss1823} (who derived the now sometimes-eponymous \emph{Gaussian function} in the context of probabilistic error analysis), Poisson~\cite{Poisson1829} (who stated a CLT for a normalized sum of random variables, and gave explicit conditions on characteristic functions for a CLT to hold), Dirichlet (who gave proofs with more rigor, including correct truncation bounds to prove errors arising in Laplace's method go to zero), and Cauchy~\cite{Cauchy1853} (who gave an updated proof of the CLT using characteristic functions). As pointed out by Fischer~\cite{Fischer2011}, it is also instructive from a historical perspective to reflect that Cauchy and Dirichlet were working at the time period when mathematics and probability were starting to move away from a discipline chiefly concerned by modelling observations of the physical world to a more abstract logic-based subject. 
A Russian school, involving mathematicians such as Hausdorff, Chebyshev, and Markov, also developed CLT-like results in the late nineteenth and early twentieth century using techniques such as the \emph{method of moments} and \emph{moment generating functions}. The first ``modern'' treatment of the CLT (as a general mathematical result not dedicated to specific applications or to illustrate analytic methods) is often considered to be work of Lyapunov~\cite{Lyapunov1900,Lyapunov1901} around the turn of the twentieth century. The term \emph{central limit theorem} was likely coined by Pólya~\cite{Polya1920} (who studied various aspects of CLTs, and when sequences of distribution functions converge) in 1920. Lévy~\cite{Levy1935} and Feller~\cite{Feller1935} gave necessary and sufficient conditions for the CLT to hold in 1935.

Modern probability theory texts typically prove the CLT using the \emph{continuity theorem} developed by Lévy~\cite{Levy1922,Levy1925} in the 1920s, which states that if the characteristic functions 
\[ \varphi_n(t) = \E\left[e^{itX_n}\right] \]
of a sequence $(X_n)$ of random variables converge pointwise to a continuous function $\varphi(t)$ then $\varphi$ is the characteristic function of some random variable $X$ and the sequence $(X_n)$ converges in distribution to $X$ (see~\cite[Theorem 26.3]{Billingsley1995} or~\cite[Theorem 3.3.6]{Durrett2010} for proofs). If $A_n$ is a sequence of independent random variables with a common characteristic function $\varphi(t)$ then the characteristic function of $X_n = A_1 + \cdots + A_n$ is $\varphi(t)^n$. Assuming that the $A_j$ have a mean 0 and a finite variance $\sigma>0$, it is possible to get an expansion
\[ \varphi(t) = 1 + (it)\E[A_1] - \frac{t^2}{2}\E[A_1^2] + o(t^2) = 1 - \frac{\sigma^2}{2}t^2 + o(t^2) \]
as $t\rightarrow0$, so the characteristic function of $X^*_n = \frac{X_n}{\sigma\sqrt{n}}$ satisfies
\[ \varphi\left(\frac{t}{\sigma\sqrt{n}}\right)^n = \left(1-\frac{t^2}{2n} + o\left(\frac{t^2}{n}\right)\right)^n \rightarrow e^{-t^2/2} \]
for each $t$ as $n\rightarrow\infty$. The CLT in the mean zero case then follows from the continuity theorem, and the result for a general mean follows from a shift of the $X_n$. 

A complete proof of the general \emph{Lindeberg-Feller CLT} can be found in~\cite[Theorem 27.2]{Billingsley1995} and~\cite[Theorem 3.4.5]{Durrett2010}. Under slightly stronger conditions (such as finiteness of the third moments of the $A_j$) the \emph{Berry–Esseen theorem}~\cite{Berry1941,Esseen1942} describes the rate of convergence between the cumulative distribution functions of the $X_n$ and the cumulative distribution function of the normal distribution (see~\cite[Theorem 3.4.9]{Durrett2010} for a modern presentation). Multivariate limit theorems, including the multivariate CLT, can be established from univariate results through the classical \emph{Cramér-Wold theorem}~\cite{CramerWold1936}: a sequence $(\bX_n)$ of $d$-dimensional random variables converges in distribution to a random variable $\bY$ if and only if the sequence of univariate random variables $\btheta \cdot \bX_n$ converges to $\btheta \cdot \bY$ for all $\btheta \in \mathbb{R}^d$ (see~\cite[Theorem 29.4]{Billingsley1995} for a proof).

\subsection{Combinatorial CLTs}
\label{sec:combCLT}

As described above, in combinatorial contexts one has a multivariate sequence defined implicitly by a multivariate generating function
\begin{equation} 
C(\bz,t) = \sum_{n \geq 0}\left(\sum_{\bi \in \N^d} c_{\bi,n}\bz^\bi \right)t^n 
\label{eq:Cgenfunc}
\end{equation}
tracking some parameter $\chi$, and wishes to extract a limit theorem for the random variables $(\bX_n)$ taking the values $\bi\in\N^d$ with probabilities
\begin{equation} p_n(\bi) = \bP\left[\chi(\sigma)=\bi \text{ when $\sigma$ has size $n$} \right] = \frac{c_{\bi,n}}{\sum_{\bi \in \N^d}c_{\bi,n}}. \label{eq:pn} \end{equation}

In the 1-dimensional case, a direct computation verifies
\[ \E[X_n] =  \frac{\sum_{i \geq 0} i \, c_{i,n}}{\sum_{i \geq 0} c_{i,n}} = \frac{[t^n]C_z(1,t)}{[t^n]C(1,t)}, \]
where $C_z$ denotes the partial derivative of $C$ with respect to $z$ and the operator $[t^n]$ extracts the coefficient of $t^n$ from a power series expansion. Similarly, the $k$th moment of $X_n$ can be expressed in terms of the first $k$ derivatives (with respect to $z$) of $C(z,t)$, and limit theorems for $X_n$ can be established by extracting coefficients from these expressions. The theory of analytic combinatorics shows how, instead of explicitly computing coefficients, their behaviour can be approximated --- to a degree often sufficient to establish limit theorems --- implicitly by studying the analytic behaviour of $C(z,t)$ in $t$ when $z \approx 1$ (for CLTs) or when $|z| \approx 1$ (for LCLTs). 

A similar approach can be used for higher-dimensional parameters, except that the translation of the analytic behaviour of $C$ to the limit behaviour of $c_{\bi,n}$ is more delicate. Classical treatments typically proceed by requiring that the \emph{coefficient slices} 
\[ \phi_n(\zz) = [t^n]C(\zz,t) = \sum_{\bi \in \N^d} c_{\bi,n}\zz^{\bi}\]
encoding the objects of size $n$ are approximated by \emph{quasi-powers}. We briefly describe some of these results, and illustrate the methods involved in their proof, to compare with the limit theorems derived using the framework of ACSV in Section~\ref{sec:ACSVandCLT}.

\begin{theorem}[{Bender and Richmond~\cite[Theorem 1]{BenderRichmond1983}}]
\label{thm:BRthm1}
Suppose that $\phi_n(\zz) \sim f(n)g(\zz)\lambda(\zz)^n$ uniformly in a neighbourhood of $\zz=\one$ where $g(\zz)$ is uniformly continuous and $\lambda(\zz)$ has a quadratic Taylor series expansion with error term $O((\sum|x_k-1|)^3)$. If the Hessian matrix $\mH$ of $\log \lambda(e^{\bss})$ at $\bss=\zero$ is non-singular then 
\[ \sup_{\ba_n} \left|\sum_{\bi \leq \ba_n} p_n(\bi) - \frac{1}{(2\pi n)^{d/2}\sqrt{\det(\mH)}} \int_{\xx \leq \ba_n} e^{-\frac{1}{2n}(\xx-n\bm)\mH^{-1}(\xx-n\bm)^T} \, d\xx \right| = o(1), \]
where $p_n(\bi)$ is the scaled coefficient defined in~\eqref{eq:pn} above, $\bm$ is the gradient of $\log \lambda(e^{\bss})$ at $\bss=\zero$, and the inequalities are taken coordinate-wise.
\end{theorem}

Theorem~\ref{thm:BRthm1} and its corollaries below were first established in the 1-dimensional case by Bender~\cite{Bender1973}. Generalizations of Theorem~\ref{thm:BRthm1} in the 1-dimensional case were given by Hwang~\cite{Hwang1996,Hwang1998,Hwang1998a} and Heuberger and Kropf~\cite{HeubergerKropf2018} (see also the treatment in Flajolet and Sedgewick~\cite[Chapter~IX]{FlajoletSedgewick2009}).

\begin{proof}[Proof Sketch]
The random variable $\bX_k$ with probability distribution function $p_n$ has characteristic function 
\[ \E\left[e^{i(\bss \cdot \bX_k)}\right] = \sum_{\bk \in \N^d} e^{i(\bss \cdot \bk)} \cdot \frac{c_{\bk,n}}{\phi(\one)} = \frac{\phi_n(e^{i\bs})}{\phi(\one)}. \]
Shifting $\bX_k$ by its mean $n\bm$ and scaling by $1/\sqrt{n}$ gives a new random variable with a characteristic function
\[ f_n(\bss) = e^{-i\sqrt{n}(\bm \cdot \bss)} \cdot \frac{\phi_n(e^{i\bss/\sqrt{n}})}{\phi_n(\one)}, \]
and the expansion
\[ \log \lambda(e^{\bss}) = \log \lambda(\one) + (\bss \cdot \bm) + \frac{\bss \mH \bss^T}{2} + \cdots \]
combines with our assumptions on $\phi_n$ to give $f_n(\bss) \sim \exp[-\frac{\bss\mH\bss^T}{2}]$ for all fixed $\bss$. The 1-dimensional random variable $\bss \cdot \bX_n$ thus satisfies a CLT for any $\bss$ by the continuity theorem, meaning $\bX_n$ satisfies a multivariate CLT by the Cramér-Wold theorem.
\end{proof}

The key insight for combinatorial CLTs is that it is often possible to prove that $\phi_n$ satisfies the assumptions of Theorem~\ref{thm:BRthm1} directly from the analytic behaviour of $C(\zz,t)$ near some of its singularities. 

\begin{corollary}[{Bender and Richmond~\cite[Corollary 1]{BenderRichmond1983}}]
\label{cor:BRcor1}
Let $C(e^{\bss},t)$ be the generating function described in~\eqref{eq:Cgenfunc}. Suppose that there exists $\epsilon,\delta>0,$ a number $q \in \Q\setminus\{-1,-2\dots\}$, and functions $A(\bss), B(\bss,t),$ and $r(\bss)$ such that 
\begin{equation} 
C(e^{\bss},t) = A(\bss)\left(1-\frac{t}{r(\bss)}\right)^{-q-1} + B(\bss,t)\left(1-\frac{t}{r(\bss)}\right)^{-q} 
\label{eq:BRcor1}
\end{equation}
where
\begin{itemize}
  \item $A(\bss)$ is continuous and non-zero,
  \item $r(\bss)$ is positive and has continuous third-order partial derivatives, and
  \item $B(\bss,t)$ is analytic and bounded
\end{itemize} 
for $\|\bss\|<\epsilon$ and $|t|<|r(\zero)|+\delta$. If the Hessian matrix of $\lambda(\bss)=1/r(\bss)$ is nonsingular at $\bss=\zero$ then the CLT in Theorem~\ref{thm:BRthm1} holds.
\end{corollary}

\begin{proof}
For any fixed $\bss$ with $\|\bss\|<\epsilon$, bounding the modulus of the Cauchy integral
\[ [t^n]B(\bss,t) = \frac{1}{2\pi i} \int_{|t| = |r(\zero)|+2\delta/3} \frac{B(\bss,t)}{t^{n+1}} \, dt \]
by the length of the curve of integration times the modulus of the integrand implies
\[ \left|[t^n]B(\bss,t) \right| \leq C \cdot (|r(\zero)|+2\delta/3)^{-n} \]
for some $C>0$. Thus, taking the coefficient of $t^n$ in~\eqref{eq:BRcor1} gives, after making use of the general binomial expansion $[z^n](1-z)^{-a}=\binom{n+a-1}{a-1}$, the bound
\begin{align*} 
\left|\phi_n(\bss) - A(\bss)\binom{n+q}{q}r(\bss)^{-n} \right| 
&= \left|[t^n]B(\bss,t)\left(1-\frac{t}{r(\bss)}\right)^{-q} \right| \\[+2mm]
& \leq C \sum_{k=0}^n \binom{k+q-1}{q-1} \; |r(\bss)|^{-k} \; (|r(\zero)|+2\delta/3)^{k-n} \\[+2mm]
& \leq C \binom{n+q-1}{q-1} |r(\bss)|^{-n} \sum_{k=0}^n (1 + |r(\bss)|\delta/3)^{-k}
\end{align*}
if $\epsilon$ is small enough such that $|r(\zero) - r(\bss)| < \delta/3$ whenever $\|\bss\|<\epsilon$. This final sum is a geometric series with ratio less than 1, so the asymptotic behaviour of binomial coefficients implies 
\[ \phi_n(\zz) = \phi_n(e^{\bss}) = \binom{n+q}{q}A(\bss)r(\bss)^{-n} \left(1 + O\left(n^{-1}\right)\right) \]
uniformly in a sufficiently small neighbourhood of $\bz=\one$, and the stated CLT follows from Theorem~\ref{thm:BRthm1}.
\end{proof}

\begin{example} 
Bender and Richmond~\cite[p. 261]{BenderRichmond1983} illustrate Corollary~\ref{cor:BRcor1} on the \emph{Tutte polynomials} $T_n(x,y)$ of \emph{wheel graphs} on $n$ vertices, which can be defined recursively by
\[T_n-(x+y+2)T_{n-1}+(xy+x+y+1)T_{n-2}-xyT_{n-3}=0\] 
for $n \geq 3$, along with the initial conditions 
\[ T_0 = xy-x-y+1, \quad T_1 = xy, \quad T_2 = x^2+y^2+xy+x+y.\] 
Solving this recurrence gives a trivariate generating function
\[C(x, y, t) = \sum T_n(x,y)t^n = \frac{(1-x+(xy-y-1)t)(1-y+(xy-x-1)t)-xyt}{(1-t)(1-(x+y+1)t+xyt^2)}\] 
which has, for positive $x$ and $y$, 
\[ t = r(x,y) = \frac{2}{x+y+1 + ((x-y)^2+2x+2y+1)^{1/2}}\] 
as the smallest root of the denominator. Then $C$ has a simple pole at $r$, as the numerator does not vanish at $t=r(x,y)$, and differentiating
\[1-(e^{s_1}+e^{s_2}+1)r\left(e^{s_1},e^{s_2}\right)+e^{s_1+s_2}r\left(e^{s_1},e^{s_2}\right)^2=0\] 
implicitly gives
\[- \grad \log r\left(e^{s_1},e^{s_2}\right) = \left(\frac{e^{s_1}-e^{s_1+s_2}r\left(e^{s_1},e^{s_2}\right)}{1+e^{s_1}+e^{s_2}-2e^{s_1+s_2}r\left(e^{s_1},e^{s_2}\right)}, \quad \frac{e^{s_2}-e^{s_1+s_2}r\left(e^{s_1},e^{s_2}\right)}{1+e^{s_1}+e^{s_2}-2e^{s_1+s_2}r\left(e^{s_1},e^{s_2}\right)}\right).\] 
At $\bss = \mathbf{0}$, further algebraic manipulation shows Corollary~\ref{cor:BRcor1} applies with $q=0$ and $A(\bss) = (1-z/r)C,$ giving a central limit theorem with 
\[\bm = \left(\frac12-\frac{1}{2\sqrt{5}}, \frac12-\frac{1}{2\sqrt{5}}\right) \quad\text{and}\quad 
\mathcal{H} = \begin{bmatrix}
\frac{3}{5\sqrt{5}} & -\frac{2}{5\sqrt{5}} \\
-\frac{2}{5\sqrt{5}} & \frac{3}{5\sqrt{5}} 
\end{bmatrix}.\] 
\end{example}

%\begin{example}
%Corollary~\ref{cor:BRcor1} applies to specific dimensions of a class of permutations with restricted cycles (see Section~\ref{sec:limitExs} for their definition). We show here how the Corollary might be applied for $d=2$, which corresponds to generating function \[f(z, t) = \frac{1}{1-t-zt^2}\] which has $r = \frac{\sqrt{4z+1}-1}{2z}$ as the smallest root of the denominator. As the numerator does not vanish at $r$, $f$ has a simple pole at $r$. By differentiating \[1-r-e^sr^2 = 0\] implicitly, we obtain \[ -\frac{\partial \log r}{\partial s} = \frac{e^sr}{1+2re^s}.\]
%At $\bss = \mathbf{0}$, Corollary~\ref{cor:BRcor1} applies with $q=0$, $A(\bss) = (1-z/r)f$ and \[\bm = \left(\frac12-\frac{1}{2\sqrt{5}}\right), \hspace{5mm} B = \begin{bmatrix}
%\frac{\sqrt{5}}{25}
%\end{bmatrix}\] giving a central limit theorem.
%\end{example}

Further CLTs of this type using the theory of singularity analysis~\cite{FlajoletOdlyzko1990} were derived (in the 1-dimensional case) by Flajolet and Soria~\cite{FlajoletSoria1990,FlajoletSoria1993} and, in the multivariate setting, by Gao and Richmond~\cite{GaoRichmond1992}. In particular, the work of Gao and Richmond allows $C(\zz,t)$ to have algebraic \emph{and} logarithmic-type singularities near $\zz=\one$. Drmota~\cite{Drmota1994a,Drmota1994} used an analytic approach (based on techniques such as the saddle-point method) to derive central limit theorems from generating functions defined implicitly by (systems of) functional equations; see also the treatment in Drmota~\cite[Section 2.2]{Drmota2009}. A detailed survey of analytic methods for the derivation of (largely 1-dimensional) CLTs is given in Flajolet and Sedgewick~\cite[Chapter~IX]{FlajoletSedgewick2009}. 

Central limit theorems of the type described in Corollary~\ref{cor:BRcor1} are derived from knowledge of the generating function $C(\zz,t)$ in a neighbourhood of $\zz=\one$. When further information about the behaviour of $C$ is known for all points with the same coordinate-wise modulus, LCLTs can also be produced.

\begin{theorem}[{Bender and Richmond~\cite[Corollary 2]{BenderRichmond1983}}]
\label{thm:BRthm2}
Let $R$ be a compact subset of $(-\infty,\infty)^d$ and suppose that there exists $\epsilon>0,$ a number $q \in \Q\setminus\{-1,-2\dots\}$, and functions $A(\bss), B(\bss,t),$ and $r(\bss)$ such that 
\begin{equation} 
C(e^{\bss},t) = A(\bss)\left(1-\frac{t}{r(\bss)}\right)^{-q-1} + B(\bss,t)\left(1-\frac{t}{r(\bss)}\right)^{-q} 
\label{eq:BRcor2}
\end{equation}
where
\begin{enumerate}[(i)]
  \item $A(\bss)$ is continuous and non-zero in an $\epsilon$-neighbourhood of $R$,
  \item $r(\bss)$ is non-zero and has continuous third-order partial derivatives in an $\epsilon$-neighbourhood of $R$,
  \item $B(\bss,t)$ is analytic and bounded for $\bss$ in an $\epsilon$-neighbourhood of $R$ and $|t|<|r(\bss)|(1+\epsilon)$,
  \item the Hessian matrix $\mH(\bss)$ of $\lambda(\bss)=1/r(\bss)$ is nonsingular for all $\bss$ in an $\epsilon$-neighbourhood of $R$, and
  \item $C(e^{\bss},t)$ is analytic and bounded whenever $|t|<|r(\textsl{Re}(\bss))|(1+\epsilon)$ and $\epsilon \leq |\textsl{Im}(\bss)_j| \leq \pi$ for all $1 \leq j \leq d$. \label{item:BRthm2}
\end{enumerate} 
Then
\[ c_{\bk,n} \sim r(\ww)^{-n} \,n^q \, \frac{e^{-\ww \cdot \bk^T} A(\ww)}{\Gamma(q+1)\sqrt{(2\pi n)^d \det \mH(\ww)}} \]
uniformly for all $\bk$ such that $\bk/n = -\nabla \log r(\ww)$ has a solution for $\ww\in R$. Furthermore, there is a local limit theorem
\[ c_{\bj,n} \sim c_{\bk,n} \cdot e^{\ww\cdot(\bk-\bj)^T} \left(\exp\left[-\frac{1}{2} \uu \mH(\ww)^{-1}\uu^{T}\right] + o(1)\right) \]
uniformly, where $\uu = (\bj-\bk)/\sqrt{n}$.
\end{theorem}

The key to establishing the LCLT in Theorem~\ref{thm:BRthm2} is Condition~\eqref{item:BRthm2}, which implies that the modulus of the singularity $t=r(\bss)$ of $C(\zz,t)$ is uniquely minimized among the points $\zz=e^\bss$ with fixed coordinate-wise modulus when $\zz$ has positive real coordinates. For instance, in one variable, information about $C(z,t)$ near the point $z=1$ is not sufficient to establish local limit theorems, one must also verify that $C(z,t)$ has no other singularities with $|z|=1$ and the same value of $|t|$. The advanced methods of ACSV, however, show that the multivariate situation is more complicated: in at least two variables there can be singularities that do not form `obstructions' to deforming domains of integration, making these singularities irrelevant to determining asymptotic behaviour. We return to this in the context of ACSV in Section~\ref{sec:ACSVandCLT} below.

\begin{proof}[Proof Sketch]
The Cauchy integral formula implies
\[ c_{\bi,n} = \frac{1}{(2\pi i)^d} \int_{|\zz|=e^{\ww}} \phi_n(\zz) \frac{d\zz}{z_1^{i_1+1}\cdots z_d^{i_d+1}} \]
for $\ww\in R$, where $|\zz|=(|z_1|,\dots,|z_d|)$. A modification of the argument used in the proof of Corollary~\ref{cor:BRcor1} shows that
\[ \phi_n(\zz) = \phi_n(e^{\bss}) = \binom{n+q}{q}A(\bss)r(\bss)^{-n} \left(1 + O\left(n^{-1}\right)\right) \]
uniformly for $\bss$ in an $\epsilon$-neighbourhood of $R$, and following the proof of Theorem~\ref{thm:BRthm1} we shift the mean of $\bX_k$ by $-n\nabla \log r(\ww)$ and scale by $1/\sqrt{n}$ to get a new random variable with characteristic function 
\begin{equation} f_n(\bss) \sim \exp\left[-\frac{\bss\mH\bss^T}{2}\right] \label{eq:BR2}\end{equation}
for certain values of $\ww$. Tracing through the definitions of the characteristic functions, and making a polar change of variables, to establish the LCLT from the Cauchy integral above it becomes sufficient to prove that
\[ \left|\int_{[-\pi\sqrt{n},\pi\sqrt{n}]^d}e^{-i(\bss\cdot\bz)}f_n(\bss)d\bss - \int_\R e^{-i(\bss\cdot\bz) - \frac{1}{2}\bss\mH\bss^T} \right| = o(1) \]
for $\ww$ given in the statement of the theorem. Near the origin (up to roughly $\sqrt{n}$-scale) it can be shown that this difference is small using~\eqref{eq:BR2}, and the second integral, which has an explicit integrand, is $o(1)$ when bounded sufficiently away from the origin. Roughly speaking, Condition~\eqref{item:BRthm2} implies that $|\phi_n(\zz)|$ is small when $\zz$ is away from the positive real axis, so the first integral is also $o(1)$ when bounded sufficiently away from the origin and the claimed limit theorem holds. 
\end{proof}

An extension of Theorem~\ref{thm:BRthm2} to functions $C(\zz,t)$ with algebraic and logarithmic singularities was given by Gao and Richmond~\cite{GaoRichmond1992}.

\section{ACSV and Multivariate CLTs}
\label{sec:ACSVandCLT}

The multivariate results in Section~\ref{sec:history} are derived from, in the words of Flajolet and Sedgewick~\cite[Page 768]{FlajoletSedgewick2009}, ``a perturbative theory of one-variable complex function theory.'' In contrast, the methods of ACSV in its modern form are based around the theory of complex analysis \emph{in several variables}. Although the ACSV framework draws from a much larger (and more recently developed) collection of mathematical techniques\footnote{For instance, Pemantle et al.~\cite{PemantleWilsonMelczer2024} includes roughly 100 pages of appendices with (highly abridged) background material, beyond the additional background material discussed in the main text.}, it is possible to package the end results in ways that allow them to be used without understanding most of this background. 

Perhaps the main advantages of ACSV are a unified framework to study the asymptotics of multivariate generating functions beyond those that exhibit the quasi-power behaviour described above (although we mainly stick to the `smooth' case with this behaviour for our treatment here) and, as alluded to previously, a deeper understanding of which singularities contribute to coefficient asymptotics (which can greatly simplify computations). The results are also explicit to the point of being completely automated for large classes of combinatorial generating functions (such as in the software package developed for this article), allow for the computation of asymptotic expansions to arbitrary order, and work under many different sets of assumptions.

\begin{rem}
There are several natural ways to take the sequence indices of $f_{\bi,k}$ to infinity and search for limit theorems: for instance, one could examine \emph{coefficient slices} with $i_1 + \cdots + i_d + k = n$ and take $n\rightarrow\infty$. For us it makes the most combinatorial sense to take the final index to infinity (as in the results above) and study coefficient behaviour in the first $d$ indices, but the methods of ACSV adapt naturally to other situations. Work extending the types of limit theorems derived by our software package is ongoing.
\end{rem}

In this section we give an overview of the methods of ACSV, and survey the explicit limit theorems they provide. Applications of these results are given in Section~\ref{sec:limitExs}. 

\subsection{Background for ACSV}
\label{sec:Background}

Let $F(\bz,t) = G(\bz,t)/H(\bz,t)$ be a ratio of complex-valued functions $G$ and $H$ analytic in a domain $\mD \subset \mathbb{C}^{d+1}$ containing the origin, and suppose that $F$ has a power series expansion 
\[ 
F(\bz,t) 
= \sum_{(\bi,k) \in \N^{d+1}} f_{\bi,k}\bz^{\bi}t^k
= \sum_{(\bi,k) \in \N^{d+1}} f_{i_1,\dots,i_d,k}z_1^{i_1} \cdots z_d^{i_d}t^k
\] 
valid in some neighbourhood of the origin (meaning, in particular, that $H(\zero,0)\neq0$). To derive limit theorems we study the asymptotic behaviour of coefficients $(f_{\br,n})_{\br\in\N^d}$ as $n\rightarrow\infty$. 
%We focus on the case when the numerator $G$ is a polynomial, and assume throughout that $H$ is not a constant to avoid the trivial case when $F$ is a polynomial.

Asymptotic arguments typically start with the Cauchy integral representation
\begin{equation}
f_{\br,n} = \frac{1}{(2\pi i)^{d+1}} \int_{\mT} F(\bz,t) \bz^{-\br} t^{-n} \frac{d\bz dt}{z_1\cdots z_d t}, 
\label{eq:CIF}
\end{equation}
where $\mT$ is any product of circles $|z_j|=|t|=\varepsilon$ sufficiently close to the origin. The methods of ACSV manipulate the domain of integration $\mT$ to convert the Cauchy integral~\eqref{eq:CIF} into something that can be asymptotically approximated. As in the more classical univariate case, this process depends heavily on the singular set of the generating function $F$. Because $F$ is a ratio, its singularities form a subset of the analytic variety $\mV = \{(\bz,t) \in \C^{d+1} : H(\bz,t)=0\}$ defined by the vanishing of the denominator $H$, and includes all points where $H$ vanishes and the numerator $G$ does not. In many applications $F$ is a rational function, in which case we may assume that $G$ and $H$ are coprime polynomials and the singular set of $F$ equals the algebraic variety $\mV$ defined by the vanishing of $H$ (a similar characterization holds for general meromorphic functions, but one must introduce the notion of \emph{coprime germs of holomorphic functions}).

Univariate meromorphic functions always have a finite set of \emph{dominant singularities} (the singularities with minimal modulus) dictating the asymptotic behaviour, and explicit expressions for their coefficient asymptotics can be determined by adding up \emph{contributions} given by each of these points. In contrast, if $F$ is not a polynomial and the dimension $d \geq 2$ then the set $\mV$ is infinite and the \emph{geometry} of $\mV$ plays a large role in determining coefficient behaviour. In order to characterize the singularities determining asymptotics, we make the following definitions. 

\begin{defn}
Let $(\bw,s) \in \C_*^{d+1} = (\C \setminus \{0\})^{d+1}$. We say that $(\bw,s)$ is
\begin{itemize}
  \item a \emph{minimal point} if $H(\bw,s)=0$ and there is no element of $\mV$ which is coordinate-wise closer to the origin, i.e., there does not exist $(\by,q)$ with $|y_j|<|w_j|$ for all $1 \leq j \leq d$ and $|q| < |s|$ such that $H(\by,q)=0$;
  \item a \emph{strictly minimal point} if it is minimal and no other point of $\mV$ has the same coordinate-wise modulus;
  \item a \emph{finitely minimal point} if it is minimal and only a finite number of points in $\mV$ have the same coordinate-wise modulus;
  \item a \emph{smooth critical point in the direction $(\br,1) \in \R_{>0}^{d+1}$} if it satisfies the system of equations
  \begin{equation} 
  \label{eq:crit}
  H(\bw,t) = 0, \qquad \frac{w_1}{r_1} H_{z_1}(\bw,s) = \frac{w_2}{r_2} H_{z_2}(\bw,s) = \cdots = \frac{w_d}{r_d} H_{z_d}(\bw,s) = t H_t(\bw,s)
  \end{equation}
  and one of these partial derivatives does not vanish (which, in fact, implies that all of the derivatives do not vanish).
\end{itemize}
\end{defn}

\begin{rem}
If $H$ and all of its partial derivatives simultaneously vanish at a point $\ww$ then either $\ww$ is a zero of $H$ \emph{with multiplicity greater than one} or $\mV$ is not a manifold near $\ww$. In the first case, $H$ can be replaced by its \emph{square-free part} near $\ww$ to determine critical points (when $H$ is a polynomial this means replacing it by the product of its distinct irreducible factors). In the second case, when $\mV$ has non-smooth points, critical points can be defined by \emph{stratifying} $\mV$ into a finite collection of smooth manifolds that `fit together nicely' and calculating critical points on each stratum. In general, if $H$ is a polynomial then the critical points on each stratum are defined by a finite collection of polynomial equalities and inequalities that can be computed automatically from $H$ (see~\cite[Section 8.2]{PemantleWilsonMelczer2024}). To simplify our presentation here we state our main results for the smooth case with zeroes of multiplicity one.
\end{rem}

\subsection{Asymptotic Results}

The earliest techniques of ACSV were derived using an explicit \emph{surgery method} that, in the smooth setting, yields limit theorems similar to the results of Bender and Richmond described above. ACSV determines asymptotic behaviour by manipulating the domain of integration $\mT$ in~\eqref{eq:CIF}, splitting it into some regions where the Cauchy integral is negligible and other regions where the integral can be approximated with analytic techniques. Roughly speaking, in the smooth setting one can push out the domain of integration $\mT$ in~\eqref{eq:CIF} to approach the set of singularities of $F$, take a residue in one variable to reduce to a $d-1$ dimensional integral lying `on the singular set,' and then (hopefully) determine asymptotics of this lower dimensional integral using the saddle-point method. Minimal points are those to which $\mT$ can be easily deformed, while critical points are those where a saddle-point analysis can be performed locally after computing a residue. 

The surgery approach to ACSV (introduced for the smooth case in~\cite{PemantleWilson2002}) applies in the presence of \emph{finitely minimal} critical points. The assumption of finite minimality allows one to make explicit residue computations by fixing the moduli of the $\bz$ variables and varying only the modulus of the $t$ variable. This makes the surgery method a fairly straightforward analogue of the techniques from univariate analytic combinatorics, however finite minimality is difficult to verify computationally and is a stronger condition than necessary. The surgery method is covered in detail in Melczer~\cite[Chapter 5]{Melczer2021}, yielding asymptotic results relying on one further quantity.

\begin{defn}
\label{def:phaseHessian}
Let $(\bw,s) \in \C_*^{d+1}$ be a smooth critical point in the direction $(\br,1)$ and suppose $H_t(\bw,s)\neq0$. The \emph{phase Hessian} $\mH(\bw,s)$ of $H$ at $(\bz,t)=(\bw,s)$ is the $d\times d$ matrix $\mH$ with entries
\begin{equation} 
\label{eq:Hess}
\mH_{i,j} = 
\begin{cases}
r_ir_j + U_{i,j} - r_jU_{i,d+1} - r_iU_{j,d+1} + r_ir_jU_{d+1,d+1} &: i \neq j \\[+3mm]
r_i +r_i^2 + U_{i,i} - 2r_iU_{i,d+1} + r_i^2U_{d+1,d+1} &: i=j
\end{cases}
\end{equation} 
where $U_{i,j} = \frac{w_iw_j H_{z_iz_j}(\bw,s)}{tH_t(\bw,s)}$ for $1 \leq i,j \leq d$ while $U_{i,d+1} = \frac{w_i H_{z_it}(\bw,s)}{H_t(\bw,s)}$ and $U_{d+1,d+1} = \frac{t H_{tt}(\bw,s)}{H_t(\bw,s)}$. The point $(\bw,s)$ is called \emph{nondegenerate} if the phase Hessian matrix $\mH(\bw,s)$ has non-zero determinant.
\end{defn}

\begin{theorem}
\label{thm:ASCV1}
Suppose that the rational function $F(\bz,t)=G(\bz,t)/H(\bz,t)$ admits a nondegenerate strictly minimal smooth critical point $(\bw,s)\in\C_*^{d+1}$ in the direction $\br\in\R_*^{d+1}$, such that $H_t(\bw,s) \neq 0$. Then for any nonnegative integer $M$ there exist computable constants $C_0,\dots,C_M$ such that
\begin{equation} 
f_{n\br} = (w_1^{r_1}\cdots w_d^{r_d}s^{r_{d+1}})^{-n} n^{-d/2} \frac{(2\pi)^{-d/2}}{\sqrt{\det(r_{d+1}\mH)}} \left(\sum_{j=0}^M C_j (r_{d+1}n)^{-j} + O\left(n^{-M-1}\right)\right), 
\label{eq:smoothAsmSimple}
\end{equation}
where $\mH=\mH(\bw,s)$ is the phase Hessian matrix and 
\[ C_0 = \frac{-G(\bw,s)}{s\, H_t(\bw,s)}. \]
The asymptotic expansion~\eqref{eq:smoothAsmSimple} holds uniformly in neighbourhoods $\mR \subset \R_*^{d+1}$ of $\br$ where there is a smoothly varying nondegenerate strictly minimal critical point such that $H_t$ does not vanish.
\end{theorem}

The fact that the asymptotic behaviour described in Theorem~\ref{thm:ASCV1} varies smoothly with the direction $\br$ under small perturbations allows (with a small amount of extra analysis) for the derivation of LCLTs. 

\begin{proposition}[{Melczer~\cite[Proposition 5.10]{Melczer2021}}]
\label{prop:LCLT}
Suppose $F(\bz,t)=G(\bz,t)/H(\bz,t)$ has a power series expansion $F(\bz,t)=\sum_{(\bi,k)\in\N^{d+1}}f_{\bi,k}\bz^{\bi}t^k$ at the origin such that $f_{\bi,k}$ is non-negative for all but a finite number of terms. Suppose further that, in some direction $(\bm,1)$, there is a strictly minimal critical point of the form $(\bone,\rho)$ for some $\rho>0$. If $H_t(\bw,\rho)$ and $G(\bw,\rho)$ are non-zero, and the phase Hessian $\mH$ of $H$ at $(\bone,\rho)$ is nonsingular, then
\begin{equation} 
\sup_{\bs\in\Z^d} n^{d/2}\left|\rho^nf_{\bs,n} - \frac{-G(\bone,\rho)}{\rho H_t(\bone,\rho)}\frac{(2\pi n)^{-d/2}}{\sqrt{\det \mH}} \exp\left[-\frac{(\bs-n\htbm)^T\mH^{-1}(\bs-n\htbm)}{2n}\right] \right|\rightarrow0.
\label{eq:LCLT2}
\end{equation}
\end{proposition}

The requirement of strict minimality in Proposition~\ref{prop:LCLT} is analogous to Condition~\eqref{item:BRthm2} in Theorem~\ref{thm:BRthm2} above. The matrix in Definition~\ref{def:phaseHessian} equals the $\mH$ in Theorems~\ref{thm:BRthm1} and~\ref{thm:BRthm2} up to sign, with the entries now determined explicitly from evaluations of partial derivatives of $H$ using~\eqref{eq:Hess}. 

\begin{rem}
Proposition 5.10 in Melczer~\cite{Melczer2021} requires $G$ and $H$ to be polynomials to simplify its presentation, however the result continues to hold for analytic functions~\cite[Remark 5.14]{Melczer2021}. Other generalizations that can be handled with the ACSV framework include (non-power series) Laurent expansions, non-smooth geometries, and non-minimal points when additional assumptions are verified.
\end{rem}

The most computationally expensive hypothesis to verify in Proposition~\ref{prop:LCLT} is strict minimality (see Melczer and Salvy~\cite{MelczerSalvy2021} for a complexity analysis of smooth ACSV methods). Generically\footnote{For instance, these properties hold for all polynomials $H$ except for those whose coefficients lie in an algebraic set determined only by the degree of $H$.} $\sing$ is smooth and the set of smooth critical point equations~\eqref{eq:crit} admits a finite set of solutions, however verifying strict minimality of a point $\ww$ requires examining whether $\sing$ intersects the (always infinite) set of points with the same coordinate-wise moduli as~$\ww$. The perturbative approach in Section~\ref{sec:combCLT} and the surgery ACSV method both require strict minimality (or finite minimality with some extra bounding), however more advanced multivariate methods show that this can be weakened.

Indeed, in the univariate setting it is the singularities of minimal modulus that contribute to dominant asymptotic behaviour, so it is tempting to assume that minimal singularities are the ones determining dominant multivariate asymptotics. In fact, as the theory of ACSV matured its methods were re-examined through more advanced mathematical frameworks, illustrating how \emph{critical points} are the singularities dictating asymptotic behaviour. The most explicit results still hold for minimal critical points, however one only needs to verify that the (generically finite) set of critical points has no other elements with the same coordinate-wise modulus as a candidate minimal critical point. Without this strengthening, algorithms to compute asymptotics (or rigorously establish limit theorems) would not terminate in reasonable time even for examples with relatively low dimension and degree.

\begin{theorem}[{Pemantle, Wilson, and Melczer~\cite[Theorem 9.12]{PemantleWilsonMelczer2024}}]
\label{thm:ASCV2}
Suppose that the rational function $F(\bz,t)=G(\bz,t)/H(\bz,t)$ admits a nondegenerate minimal smooth critical point $(\bw,s)\in\C_*^{d+1}$ in the direction $\br\in\R_*^{d+1}$, such that $H_t(\bw,s) \neq 0$ and no other critical point has the same coordinate-wise modulus as $(\bw,s)$. Then the conclusions of Theorem~\ref{thm:ASCV1} hold, where the expansion~\eqref{eq:smoothAsmSimple} now holds uniformly over neighbourhoods where there is a smoothly varying nondegenerate minimal critical point such that $H_t$ does not vanish and no other critical point has the same coordinate-wise modulus.
\end{theorem}

Theorem~\ref{thm:ASCV2} was first discussed in Pemantle and Baryshnikov~\cite{BaryshnikovPemantle2011} using \emph{cones of hyperbolicity}. To briefly summarize, if $(\bz,t)$ is a minimal point then the tangent plane to the modified function $H(e^{x_1},\dots,e^{x_d},e^q)$ at $(\log(\bz),\log(t))$ is defined by a normal vector that is a multiple of a real vector $\vv$. If $(\bz,t)$ is not critical then $\vv$ is not parallel to the direction vector $\rr$ and this can be used to locally deform a domain of integration near $(\bz,t)$ into a region of complex space where it can be bounded and shown to be negligible. The framework of hyperbolic cones shows that these local deformations can be done in a consistent manner (away from critical points) and also generalizes to non-smooth cases. 

\begin{example}
The rational function
\[ F(x,t) = \frac{1}{(1+x)(2-x-t)} \]
admits $(x,t) = (1,1)$ as a minimal critical point in the direction $\br=(1,1)$, however this point is not finitely minimal as $(-1,t)$ is a singularity for any $t \in \C$, and none of Theorem~\ref{thm:BRthm2}, Theorem~\ref{thm:ASCV1}, or Proposition~\ref{prop:LCLT} directly apply. Since there are no other critical points with the same coordinate-wise modulus as $(1,1)$, Theorem~\ref{thm:ASCV2} does apply and we can compute 
\[ f_{n,n} = n^{-1/2}\left(\frac{1}{4\sqrt{\pi}} + O\left(n^{-1}\right)\right). \]
In fact, the local central limit theorem 
\[ \sup_{s\in\Z} n^{1/2}\left|f_{s,n} - \frac{1}{4}\frac{1}{\sqrt{\pi n}} \exp\left[-\frac{(s-n)^2}{4n}\right] \right|\rightarrow0 \]
holds.
\end{example}

When dealing with combinatorial generating functions, the importance of critical points can be seen directly. Indeed, the idea underpinning cones of hyperbolicity (that the tangent space near any smooth minimal point has a normal vector that is a multiple of a real vector) implies that every minimal point is critical in some direction. When the series under consideration has nonnegative coefficients, this implies that any minimal point with the same coordinate-wise modulus as a critical point with positive coordinates is also critical~\cite[Corollary 5.5]{Melczer2021}, giving the following result on which we base our algorithm for automatically finding and verifying LCLTs.

\begin{proposition}
\label{prop:LCLT2}
Suppose $F(\bz,t)=G(\bz,t)/H(\bz,t)$ has a power series expansion $F(\bz,t)=\sum_{(\bi,k)\in\N^{d+1}}f_{\bi,k}\bz^{\bi}t^k$ at the origin such that $f_{\bi,k}$ is non-negative for all but a finite number of terms. Suppose further that, in some direction $(\bm,1)$, there is a minimal critical point of the form $(\bone,\rho)$ for some $\rho>0$ and no other critical point has the same coordinate-wise modulus. If $H_t(\bw,\rho)$ and $G(\bw,\rho)$ are non-zero, and the phase Hessian $\mH$ of $H$ at $(\bone,\rho)$ is nonsingular, then the LCLT~\eqref{eq:LCLT2} holds.
\end{proposition}

Going even further, Baryshnikov, Pemantle, and Melczer~\cite{BaryshnikovMelczerPemantle2022} show how asymptotic behaviour can be characterized in the absence of minimal critical points using techniques from \emph{stratified Morse theory}. This requires introducing the notion of \emph{critical points at infinity} and gives asymptotics as a linear combination of asymptotic expansions with (generally unknown) integer coefficients, so here we stick to the explicit case of (not necessarily strictly) minimal critical points covered by Theorem~\ref{thm:ASCV2}.

Finally, although it is beyond the scope of the present article, we give a non-smooth example from forthcoming work that illustrates the differences from the smooth cases discussed above.

\begin{example}
Consider the generating function 
\[ F(x,t) = \frac{6}{(1-3t(1+x)) \, (1-2t(2+x))}. \]
The methods of ACSV for \emph{transverse multiple points} (see~\cite[Chapter 9]{Melczer2021} or~\cite[Chapter 10]{PemantleWilsonMelczer2024}) imply that the asymptotic behaviour of $f_{\lambda n, n}$ varies with $\lambda \in (0,1)$ in a computable manner. Indeed, if $0 < \lambda < 1/3$ or $1/2 < \lambda < 1$ then the dominant asymptotic behaviour of $f_{\lambda n,n}$ is dictated by a smooth minimal critical point vanishing on only one of the denominator factors of $F$, and $f_{\lambda n,n} \rightarrow 0$ exponentially quickly. In contrast, when $\lambda \in (1/3,1/2)$ then the dominant asymptotic behaviour is determined by the non-smooth minimal critical point $(1,1)$ where both denominator factors vanish, and the methods of ACSV imply $f_{\lambda n,n} \rightarrow 6$. 

\begin{figure}[h!]
\centering
\includegraphics[width=0.32\linewidth]{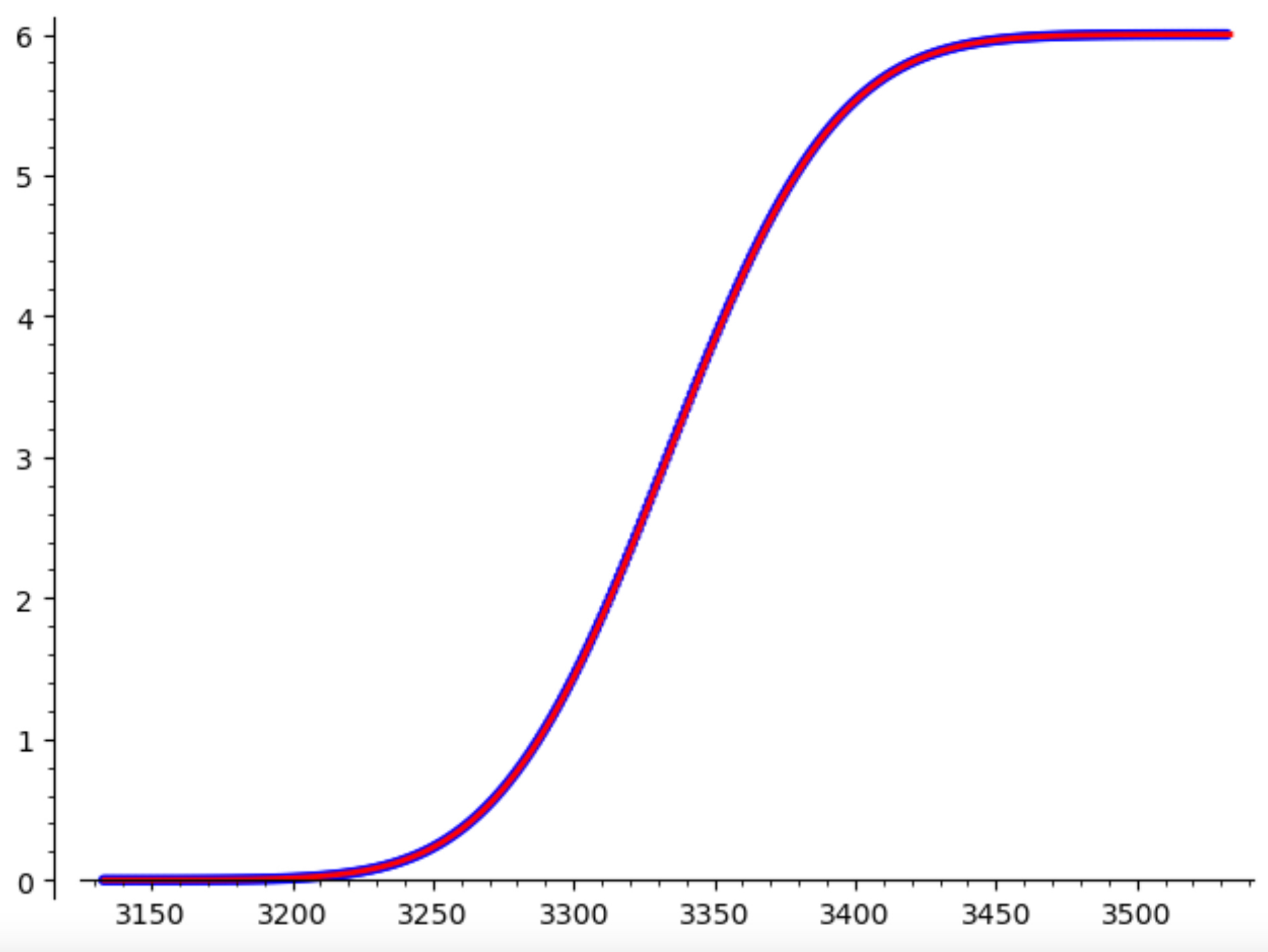} \;
\includegraphics[width=0.32\linewidth]{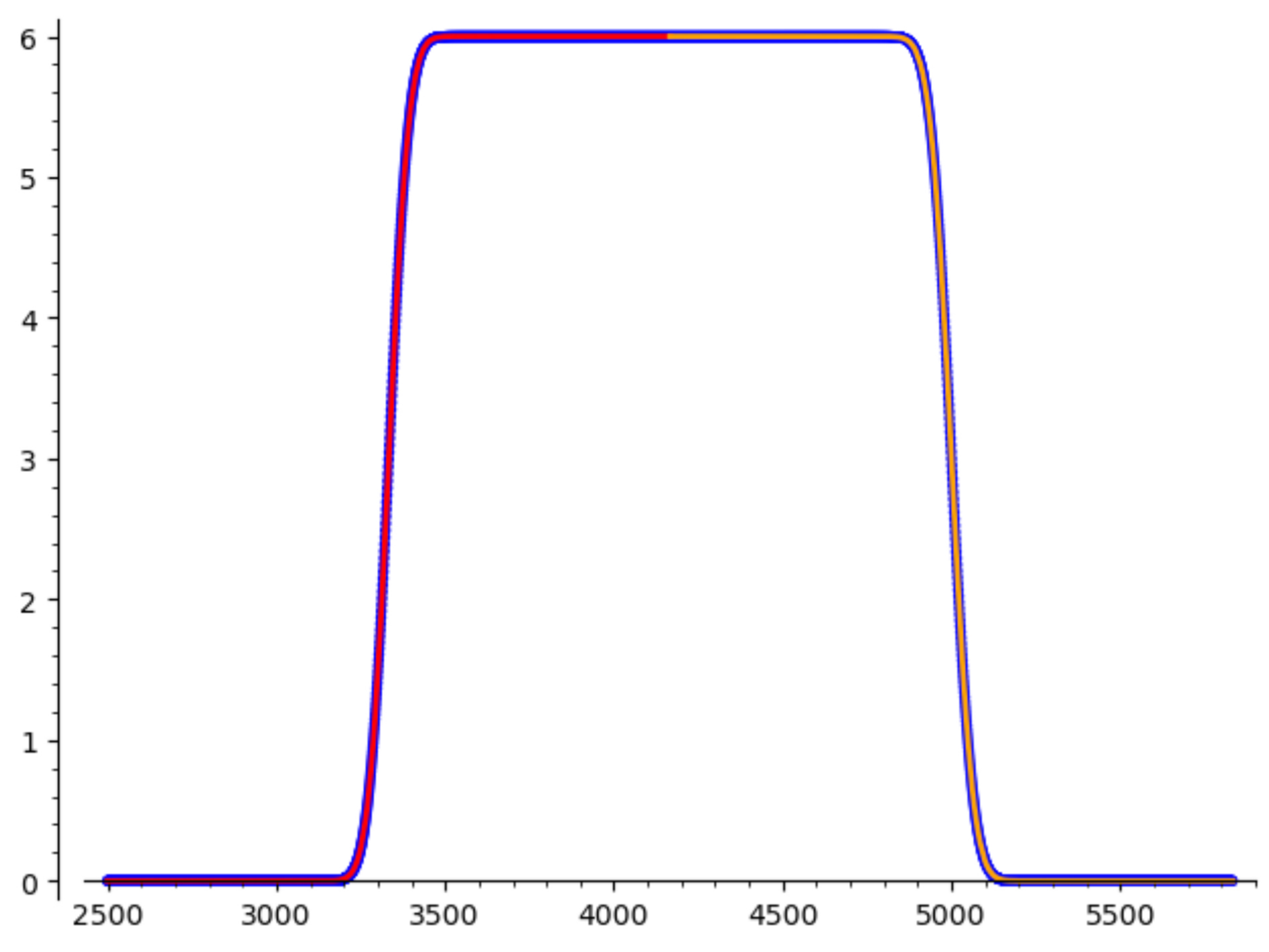} \;
\includegraphics[width=0.32\linewidth]{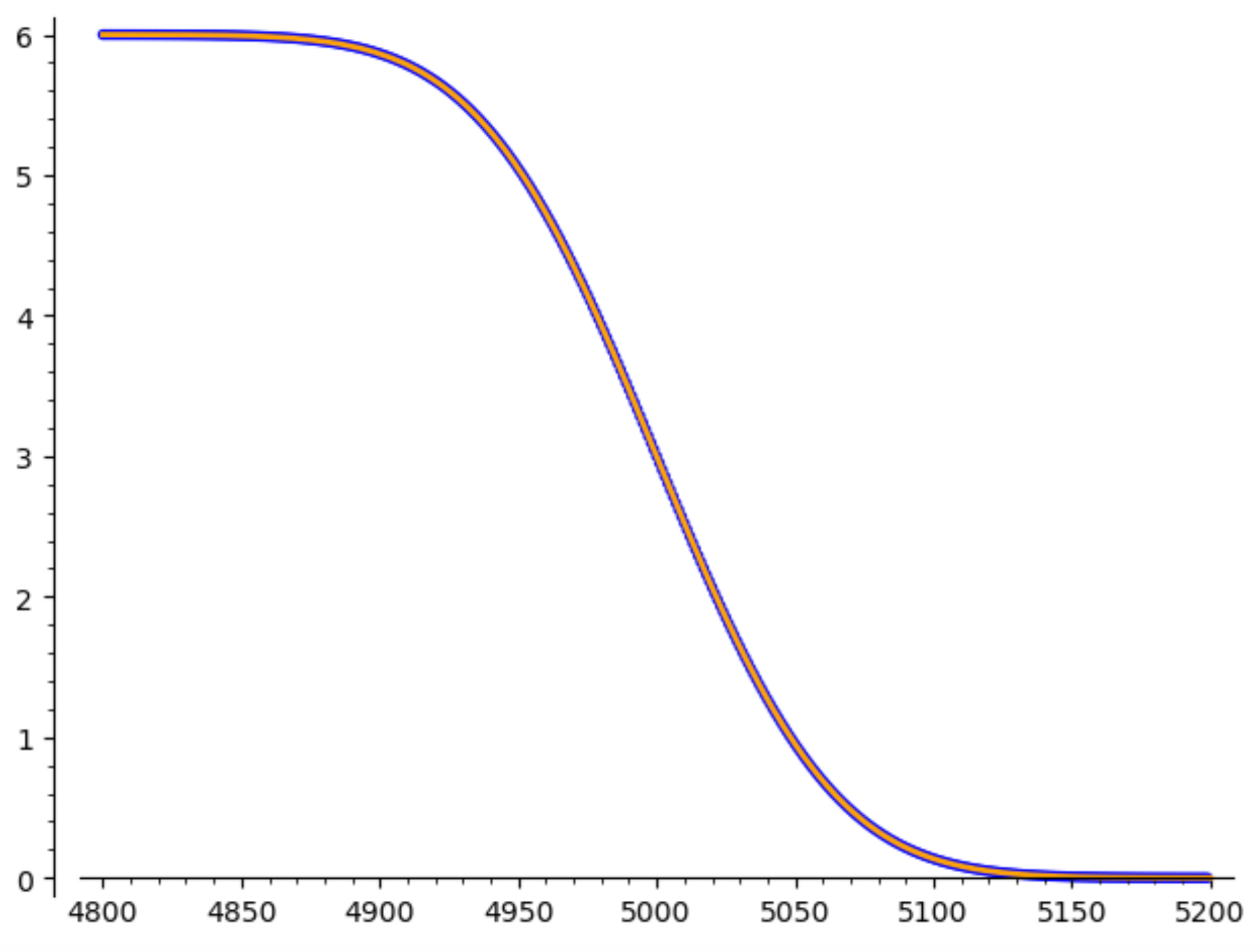}
\caption{The coefficients $c(k) = f_{k,N}$ (in blue) compared to their limiting behaviour (in red and orange) when $N=1000$. The transitions in behaviour near $N/3$ and $N/2$ are shown on intervals of length $2\sqrt{N}$ in the left and right plots.}
\label{fig:NonSmooth}
\end{figure}

Thus, unlike the smooth case where a central limit theorem is obtained, here we have a range of indices where the same minimal critical point determines asymptotics, giving a limiting distribution with a flat plateau (see the middle plot in Figure~\ref{fig:NonSmooth}). A minor modification of the arguments in~\cite[Section 6.2]{BaryshnikovMelczerPemantle2024a} -- developed there only for generating functions whose denominator factors are linear -- shows how to capture the transition between the different limiting regimes, which occurs on a square-root scale. Indeed, if
\[ \Phi(t) = \frac{2}{\sqrt{\pi}}\int_0^t e^{-s^2}ds \]
is the \emph{Gaussian error function} then
\[
f_{n/3 + t\sqrt{n}, \; n} \sim 3 + 3\Phi(3t/2) \quad \text{ and } \quad
f_{n/2 + t\sqrt{n}, \; n} \sim 3 - 3\Phi(\sqrt{2} \, t)
\]
when $t$ grows sufficiently slower than $\sqrt{n}$ (for instance, when $t$ bounded). See Figure~\ref{fig:NonSmooth} for an illustration.
\end{example}

\subsection{Verifying Minimality}

Because verifying criticality is easier than verifying minimality, most ACSV algorithms for asymptotics fix a direction, compute critical points in this direction, and then study the critical points to determine which are minimal points. In contrast, because Proposition~\ref{prop:LCLT} requires a minimal critical point of the form $(\bone,t)$, to prove an LCLT it is often easiest to use~\eqref{eq:crit} to discover a direction $\br = (\bm,1)$ with critical points of this form and then verify the required conditions. 

Determining minimality is easier for points with positive coefficients when $F(\bz)$ has only a finite number of non-negative coefficients, as it does under our assumptions. The following result should be seen as a multivariate generalization of the well-known Vivanti-Prinsheim theorem in the univariate case, and the approach to proving an LCLT that it suggests when combined with Proposition~\ref{prop:LCLT2} is summarized in Figure~\ref{fig:schema2}. 

\begin{lemma}[{Melczer~\cite[Lemma 5.7]{Melczer2021}}]
\label{lemma:minimal}
Suppose $F(\bz,t)=G(\bz,t)/H(\bz,t)$ has a power series expansion at the origin with (at most) a finite number of negative coefficients. Then $(\bw,\rho) \in \R_{>0}^d$ is minimal if the line segment from the origin to $(\bw,\rho)$ contains no roots of $H(\bz,t)$, i.e., if
\[ H(sw_1,\dots,sw_d,s\rho) \neq 0 \text{ for all } s\in(0,1). \]
\end{lemma}

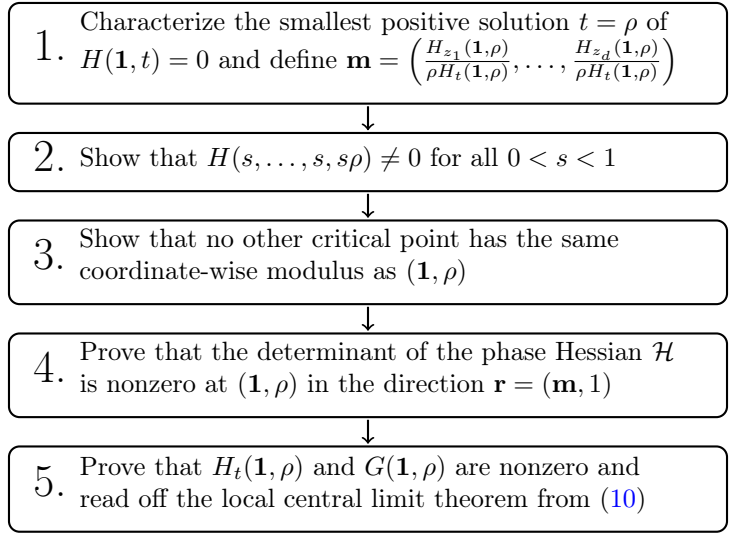
\begin{figure}
\centering
\tikzstyle{vspecies}=[rectangle, minimum size=0.5cm,text width = 26.5em,draw=black,fill=white, rounded corners,thick]
\begin{tikzpicture}[auto, outer sep=1pt, node distance=1.5cm]
\node [vspecies] (B) {\hspace{0.035in} {\LARGE 1.} \begin{minipage}{0.87\linewidth} Characterize the smallest positive solution \mbox{$t=\rho$} of $H(\bone,t)=0$ and define $\bm = \left(\frac{H_{z_1}(\bone,\rho)}{\rho H_t(\bone,\rho)},\dots,\frac{H_{z_d}(\bone,\rho)}{\rho H_t(\bone,\rho)}\right)$  \vspace{0.05in} \end{minipage}} ;
\node [vspecies, below = 0.32cm of B] (C) {\hspace{0.035in} {\LARGE 2.} \begin{minipage}{0.87\linewidth} Show that $H(s,\dots,s,s\rho)\neq 0$ for all $0<s<1$ \vspace{0.05in} \end{minipage}} ;
\node [vspecies, below = 0.32cm of C] (C2) {\hspace{0.035in} {\LARGE $3$.} \begin{minipage}{0.87\linewidth} Show that no other critical point has the same coordinate-wise modulus as $(\bone,\rho)$ \vspace{0.05in} \end{minipage}} ;
\node [vspecies, below = 0.32cm of C2] (D) {\hspace{0.035in} {\LARGE 4.} \begin{minipage}{0.87\linewidth} Prove that the determinant of the phase Hessian $\mH$ is nonzero at $(\bone,\rho)$ in the direction $\br = (\bm,1)$ \vspace{0.05in} \end{minipage}} {};
\node [vspecies, below = 0.32cm of D] (E) {\hspace{0.035in} {\LARGE 5.} \begin{minipage}{0.87\linewidth} Prove that $H_t(\bone,\rho)$ and $G(\bone,\rho)$ are nonzero and read off the local central limit theorem from~\eqref{eq:LCLT2}  \vspace{0.05in} \end{minipage}} ;
\draw [<-,thick] (C) --  node {} (B) ;
\draw [<-,thick] (C2) --  node {} (C) ;
\draw [<-,thick] (D) --  node {} (C2) ;
\draw [<-,thick] (E) --  node {} (D) ;
\end{tikzpicture}
\caption{A schema to prove LCLTs using Proposition~\ref{prop:LCLT2}.}
\label{fig:schema2}
\end{figure}

Although proving strict minimality is difficult in general, there is one case arising in some combinatorial applications where it is automatic.

\begin{defn}
A power series $S(\bz) = \sum_{\bn\in\mathbb{N}^d}p_{\bn}\bz^{\bn}$ is called \textit{aperiodic} if every element of $\mathbb{Z}^d$ can be written as an integer linear combination of the exponents $\{\bn\in\mathbb{N}^d:p_{\bn}\neq0\}$ appearing in $S$.
\end{defn}

\begin{proposition}[{Melczer~\cite[Proposition 5.5]{Melczer2021}}]
\label{prop:strictmin}
Suppose $F(\bz) = G(\bz)/H(\bz)$ is a ratio of functions $G$ and $H$. If $H(\bz) = 1 - S(\bz)$ for some aperiodic power series $S$ with non-negative coefficients then every minimal point that is within the domain of convergence of the power series expansion of $F(\bz)$ is strictly minimal and has positive real coordinates.
\end{proposition}

\begin{proof}
Suppose that $\bw$ is a minimal point and for each $1 \leq j \leq d$ write $w_j = x_je^{i\theta_j}$ with $x_j > 0$ and $\theta_j \in \mathbb{R}$. Let $s_{\bn}$ denote the coefficient of $z^{\bn}$ in $S(\bz)$. Then
\begin{align*}
1 = |S(\bw)| = \left|\sum_{\bn \in \mathbb{N}^d}s_{\bn}\bx^{\bn}e^{i(\bn\theta)}\right| \leq \sum_{\bn \in \mathbb{N}^d}s_{\bn}\bx^{\bn}
\end{align*}
since $\bw$ is within the domain of convergence of $G$ and $H$ and hence $|S(\bw)| \leq S(|w_1|, \ldots, |w_d|)$.
\end{proof}

\section{A Family of Limit Theorems}
\label{sec:limitExs}

In this section, and the next one, we illustrate how to apply the approach of Figure~\ref{fig:schema2} to prove LCLTs. We begin with the example that originally motivated our work.

\begin{defn}
For any $d \in \N$, let $\mF_d(n)$ be the set of permutations $\sigma$ on $\{1,\dots,n\}$ such that $i - d \leq \sigma(i) \leq i + 1$ for all $i$. Let $\mathcal{F}_d$ be the union of sets $\mF_d(n)$ for all $n \in \mathbb{N}$. 
\end{defn}

Note that every element of $\mF_d$, when written in disjoint cycle notation, has cycles of length at most $d+1$.

\begin{proposition}[{Chung et al. \cite[Theorem 1]{ChungDiaconisGraham2021}}]
\label{prop:ratGF}
The number of permutations in $\mF_d(n)$ with $i_k$ cycles of length $k$ equals the coefficient $[z^{i_1}_1\dotsb z^{i_{d+1}}_{d+1}t^n]F(\bz, t)$ in the power series expansion of the rational function
\begin{align*}
    F(\bz,t) = F(z_1,\dots,z_{d+1}, t) = \frac{1}{1-z_1t-z_2t^2-\cdots-z_{d+1}t^{d+1}}.
\end{align*}
\end{proposition}

Chung et al. \cite{ChungDiaconisGraham2021} prove Proposition 2 by considering the set of perfect matchings of the graph associated with $\mF_d$ and using this to find an explicit recurrence that can be manipulated. One of the motivations for their study of this family is a relationship to the determination of sample sizes required for sequential importance sampling of certain random perfect matchings in classes of bipartite graphs.

\begin{conjecture}[{Chung et al. \cite[Page 45]{ChungDiaconisGraham2021}}]
\label{conj:LCLT}
For fixed $d$ the joint limiting distribution of the number of $k$-cycles approaches a multivariate normal distribution as $n\rightarrow\infty$.
\end{conjecture}

Experimentally checking the properties of Proposition~\ref{prop:LCLT2} in low dimension using a computer algebra system, we were initially surprised to find that some of the properties did not hold! To better understand the behaviour of the coefficients we thus plotted the coefficients of $[t^{150}]F(\bz,t)$ for $d=1$, shown on the left of Figure~\ref{fig:plots}.

\begin{figure}[h!]
\centering
\includegraphics[scale=0.42]{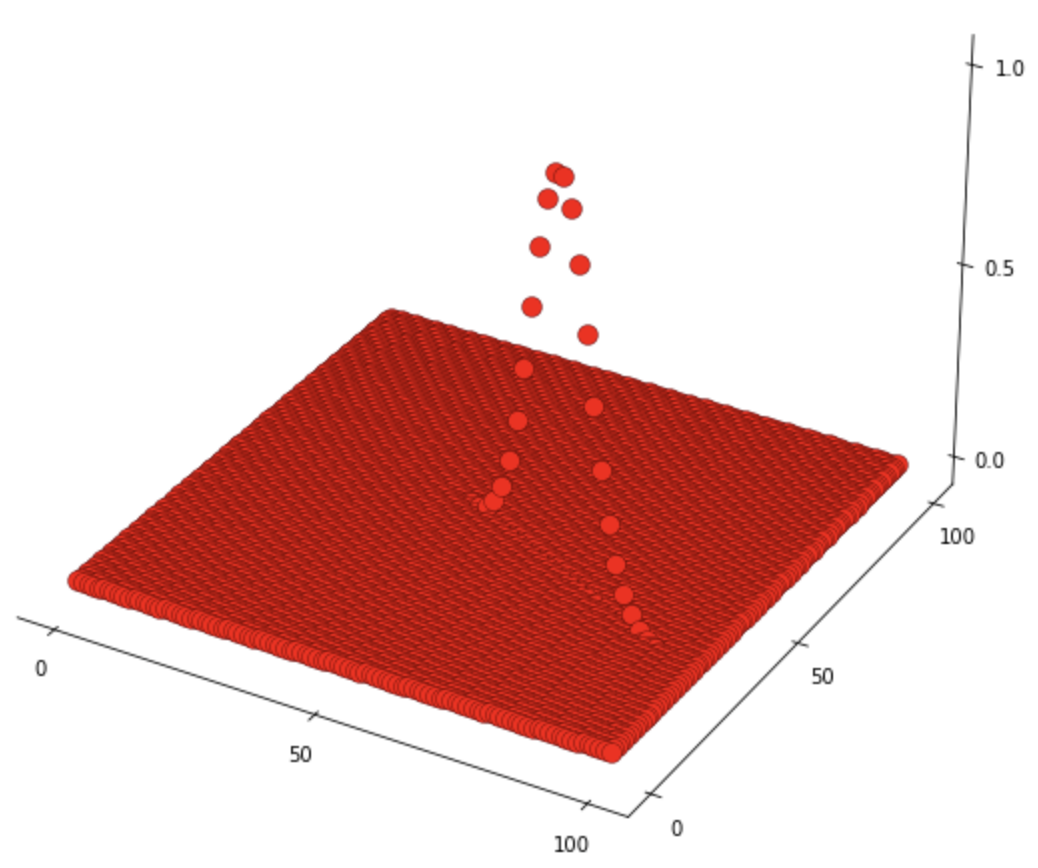}
\includegraphics[scale=0.42]{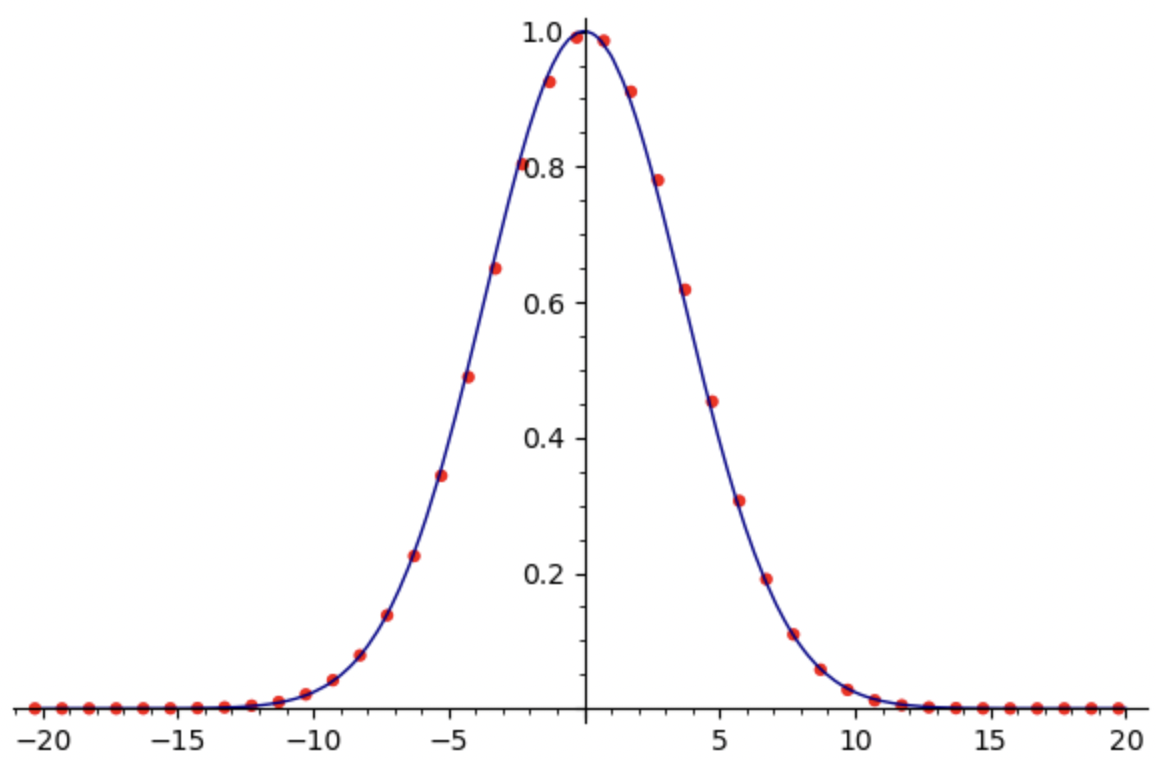}
\caption{\emph{Left:} The coefficients of $[t^{150}]F(\bz,t)$ when $d=1$, divided by the maximum coefficient and shifted to put the maximum at the origin. \emph{Right:} The coefficients of $[t^{150}]F(1,z_2,t)$ compared to their limiting normal distribution and shifted to put the maximum at the origin.}
\label{fig:plots}
\end{figure}

Figure~\ref{fig:plots} shows the problem establishing a limit theorem on the coefficients of $F(\bz,t)$: the coefficients approach a normal distribution, but the distribution is supported on a $d$-dimensional slice of $\R^{d+1}$. Indeed, if the size $n$ of one of our restricted permutations is fixed, and the number $i_k$ of $k$ cycles it contains is specified for all $k \geq 2$, then its number of one cycles (i.e., fixed points) is uniquely determined as $i_1 = n - 2i_2 - \cdots - (d+1)i_{d+1}$. We thus prove Conjecture~\ref{conj:LCLT} true by setting $z_1=1$ and applying the techniques of ACSV to the coefficients of $F(1,z_2,\dots,z_d,t)$. 

\begin{theorem}
\label{thm:LCLT}
Let $h(t) = 1 - t - \cdots - t^{d+1}$ and let $\rho>0$ be the smallest positive root of $h(t)$. As $n\rightarrow\infty$, the maximum coefficient of $[t^n]F(\bz, t)$ as a polynomial in $z_2,\dots,z_{d+1}$ approaches
\begin{align*}
  A_n =  \frac{\rho^{-n} n^{-d/2}}{-\rho h'(\rho)(2\pi)^{d/2}} \sqrt{\frac{(1+2\rho+\dots+(d+1)\rho^d)^{d+2}}{(1+\rho+\dots+\rho^d)\rho^{\frac{d(d+1)}{2}}}}.
\end{align*}
Furthermore, 
\begin{align}
\label{eq:LCLT}
     \sup_{s_2,\dots,s_{d+1} \in \mathbb{N}}\left|\frac{[z^{s_2}_2\cdots z^{s_{d+1}}_{d+1}t^n]F(1, z_2,\dots,z_{d+1},t)}{A_n} - v_n(s_2,\dots,s_{d+1})\right| \rightarrow 0
\end{align}
where 
\begin{itemize}
\item $v_n(\bss) = \exp\left[-\frac{(\bss-n\bm)\mH^{-1}(\bss-n\bm)^T}{2n}\right]$,
\item $\bm = \left(-\frac{\rho^2}{h'(\rho)}, \ldots, -\frac{\rho^{d+1}}{h'(\rho)}\right)$,
\item and $\mH$ is the non-singular matrix with entries
\begin{equation*} 
\mH_{i,j} = \begin{cases}
    \frac{\rho^{i+j+1}h''(\rho)-\rho^{i+j}(1+i+j)h'(\rho)}{h'(\rho)^3} & i \neq j \\[+2mm]
    \frac{\rho^{i+j+1}h''(\rho)-\rho^{i+j}(1+i+j)h'(\rho)-\rho^ih'(\rho)^2}{h'(\rho)^3} & i = j
    \end{cases} \;.
\end{equation*}
\end{itemize}
\end{theorem}

Theorem~\ref{thm:LCLT} follows directly from Proposition~\ref{prop:LCLT2} after some direct calculations, the most difficult of which is computing a closed form for the determinant of the Hessian matrix $\mathcal{H}$. We complete these calculations in Section~\ref{sec:GenForm}, using a guess-and-check method for symbolic determinants which is a useful tool for establishing limit theorems such as these with parameterized dimension. In fact, we will prove the following more general result.

\begin{theorem}[Main Theorem]
\label{thm:LCLTgen}
Let $F(\bz, t) = \frac{G(\bz, t)}{H(\bz,t)}$ be a ratio of functions where
$$H(\bz, t) = 1 - q(t) - \sum_{k=1}^dq_k(t)z_k,$$
let $P(t) = H(\bone, t)$, and let $\rho$ be the smallest positive root of $P(t)$. Suppose that
\begin{itemize}
\item each of the $q_i(t)$ is a non-zero polynomial vanishing at the origin and $q(t)$ is a complex-valued analytic function for $|t| \leq \rho$ that vanishes at the origin,
\item the power series expansion of $S(\bz,t) = 1-H(\bz, t)=q(t) + \sum_k q_k(t)z_k$ at the origin has non-negative coefficients, 
\item the exponents appearing in the power series $q(t)$ have greatest common divisor $1$, and
\item $G(\bone, \rho)$ is non-zero.
\end{itemize}
As $n \rightarrow \infty$, the maximum coefficient of $[t^n]F(\bz, t)$ as a polynomial in $z_1, \ldots, z_d$ approaches
\[
A_n = \rho^{-n}n^{-d/2} \frac{G(\bone, \rho)}{-\rho P'(\rho)(2\pi)^{d/2}\sqrt{\det \mH}},
\]
where $\mH$ is the non-singular $d \times d$ matrix  
\begin{equation} 
\mH_{i,j} = \begin{cases}
    \displaystyle \frac{\rho q_i(\rho)q_j(\rho)P''(\rho)-(q_j(\rho)q'_i(\rho)\rho+q_i(\rho)q'_j(\rho)\rho-q_i(\rho)q_j(\rho))P'(\rho)}{\rho^2P'(\rho)^3} & i \neq j \\[+4mm]
    \displaystyle \frac{\rho q_j(\rho)^2P''(\rho)-(2q_j(\rho)q'_j(\rho)\rho-q_j(\rho)^2)P'(\rho) - q_j(\rho)\rho P'(\rho)^2}{\rho^2P'(\rho)^3} & i = j
    \end{cases} \;
\label{eq:Hgen}
\end{equation}
whose determinant
\begin{equation}
\det \mH 
= \frac{(-1)^d\left(\prod_{k=1}^{d}q_k(\rho)\right)\left[(q(\rho)-1)\left(\rho P''(\rho) + P'(\rho)+\rho\sum_{k=1}^{d}\frac{q'_k(\rho)^2}{q_k(\rho)}\right) + q'(\rho)^2\rho\right]}{P'(\rho)^{d+2}\rho^{d+1}}
\label{eq:Hgendet}
\end{equation}
is non-zero. Furthermore, 
\[
\sup_{s_1,\ldots, s_d \in \N} \left|\frac{[z_1^{s_1}\cdots z_d^{s_d}t^n]F(z_1, \ldots, z_d, t)}{A_n} - v_n(s_1, \ldots, s_d)\right| \rightarrow 0
\]
where
$$ v_n(\bs) = \exp\left[-\frac{(\bs-n\bm)\mH^{-1}(\bs-n\bm)^T}{2n}\right]
\quad\text{ for } \quad
\bm = \left(\frac{-q_1(\rho)}{\rho P'(\rho)}, \ldots, \frac{-q_d(\rho)}{\rho P'(\rho)}\right).$$
\end{theorem}

\begin{rem}
The Sage package accompanying this article automatically proves central limit theorems for all explicit rational functions satisfying the conditions of Proposition~\ref{prop:LCLT2}, which includes all those satisfying the conditions of Theorem~\ref{thm:LCLTgen}. 
\end{rem} 

Although the form of $H$ in Theorem~\ref{thm:LCLTgen} might seem restrictive, generating functions of this form appear quite frequently when tracking parameters in combinatorial classes using, for instance, the \emph{symbolic method framework} described in Flajolet and Sedgewick~\cite{FlajoletSedgewick2009}. We end this section by describing some other combinatorial limit theorems it captures.

\subsubsection*{Strings with Tracked Letters}
Let $\mA=\{a_1,\dots,a_\ell\}$ be an alphabet with $\ell$ letters and let $\Omega = \{\omega_1,\dots,\omega_d\} \subset \mA$ be a subset of $d$ letters we wish to track. The multivariate generating function enumerating such strings is 
$$F(\bz,t) = \frac{1}{1-(z_1 + z_2 + \ldots + z_d)t-(\ell-d)t},$$
which trivially satisfies the hypotheses of Theorem~\ref{thm:LCLTgen} when $\ell > d$ (if $\ell=d$ then all coefficients live in a $d$-dimensional hyperplane of $\mathbb{R}^{d+1}$, because adding the number of occurrences of each letter gives the length of the string). Thus, as $n \rightarrow \infty$, the maximum coefficient of $[t^n]F(\bz, t)$ as a polynomial in $z_1, \ldots, z_d$ approaches
\[
A_n = \ell^{n}n^{-d/2} \frac{1}{ (2\pi)^{d/2}\sqrt{\frac{\ell - d}{\ell^{d+1}}}}
\]
and
\[
\sup_{s_1,\ldots, s_d \in \N} \left|\frac{[z_1^{s_1}\cdots z_d^{s_d}t^n]F(z_1, \ldots, z_d, t)}{A_n} - \exp\left[-\frac{(\bss-n\bm)\mH^{-1}(\bss-n\bm)^T}{2n}\right]\right| \rightarrow 0,
\]
where $\bm = \left(\frac{1}{\ell}, \ldots, \frac{1}{\ell}\right)$ and $\mH$ is the non-singular $d \times d$ matrix with off-diagonal entries $-\frac{1}{\ell^2}$ and diagonal entries $\frac{\ell - 1}{\ell^2}$. Note that tracking the number of $1$s in binary strings (where $\ell=2$ and $d=1$) recovers the classical central limit theorem,
$$ \binom{n}{s} = [z^st^n]\frac{1}{1-(1+z)t} \approx 2^n \sqrt{\frac{2}{\pi n}} e^{-(n-2s)^2/2n}. $$

\subsubsection*{Compositions with Tracked Summands}
\label{sec:Comps}

Fix a positive integer $d$ and recall that an \emph{(integer) composition} of size $n$ is an ordered tuple of positive integers which sum to $n$. If $\CC$ is the class of compositions, enumerated by size and the number of times each element of $\{1,\dots,d\}$ occurs, then $\CC$ has the multivariate generating function
$$\CC(\bz, t) = \frac{1}{1-S(\bz,t)} = \frac{1}{1-z_1t - z_2t^2 - \cdots - z_dt^d - \frac{t^{d+1}}{1-t}},$$ 
where
\begin{align*}
S(\bz, t) 
&= z_1t + z_2t^2 + \cdots + z_dt^d + \sum_{k \geq d+1} t^k \\[+2mm]
&= z_1t + z_2t^2 + \cdots + z_dt^d + \frac{t^{d+1}}{1-t}
\end{align*}
is the multivariate generating function of positive integers where $z_k$ tracks the number of occurrences of $k$. Once again, it is easy to verify the hypotheses of Theorem~\ref{thm:LCLTgen} hold with $\rho=1/2$, the smallest positive root of $P(t) = H(\bone, t) = 1 - \frac{t}{1-t}$. Figure~\ref{fig:compplots} illustrates the corresponding local central limit theorem when $d = 1$ and $d = 2$.

\begin{theorem}
\label{thm:LCLTcomp}
As $n\rightarrow\infty$ the maximum coefficient of $[t^n]C(\bz, t)$ as a polynomial in $z_1,\dots,z_d$ approaches
\begin{align*}
  A_n = 2^n n^{-d/2}\frac{2^{\frac{d^2}{4} + \frac{7d}{4} + 1}}{2(2\pi)^{d/2} \sqrt{(d^2 + 4d + 6)2^d - 2}},
\end{align*}
and
\begin{align*}
     \sup_{s_1,\dots,s_{d} \in \mathbb{N}}\left|\frac{[z^{s_1}_1\cdots z^{s_{d}}_{d}t^n]\CC(z_1,\dots,z_d,t)}{A_n} - \exp\left[-\frac{(\bss-n\bm)\mH^{-1}(\bss-n\bm)^T}{2n}\right]\right| \rightarrow 0
\end{align*}
where $\bm = \left(\frac{1}{4}, \frac{1}{8}, \ldots, \frac{1}{2^{d+1}}\right)$ and $\mH$ is the $d \times d$ matrix with off-diagonal entries $\mH_{i,j} = -2^{-i - j - 2}(i + j - 3)$ and diagonal entries $\mH_{j,j} = 2^{-2(j+1)}(2^{j+1}-2j+3)$.
\end{theorem}

\begin{figure}
\centering
\includegraphics[scale=0.42]{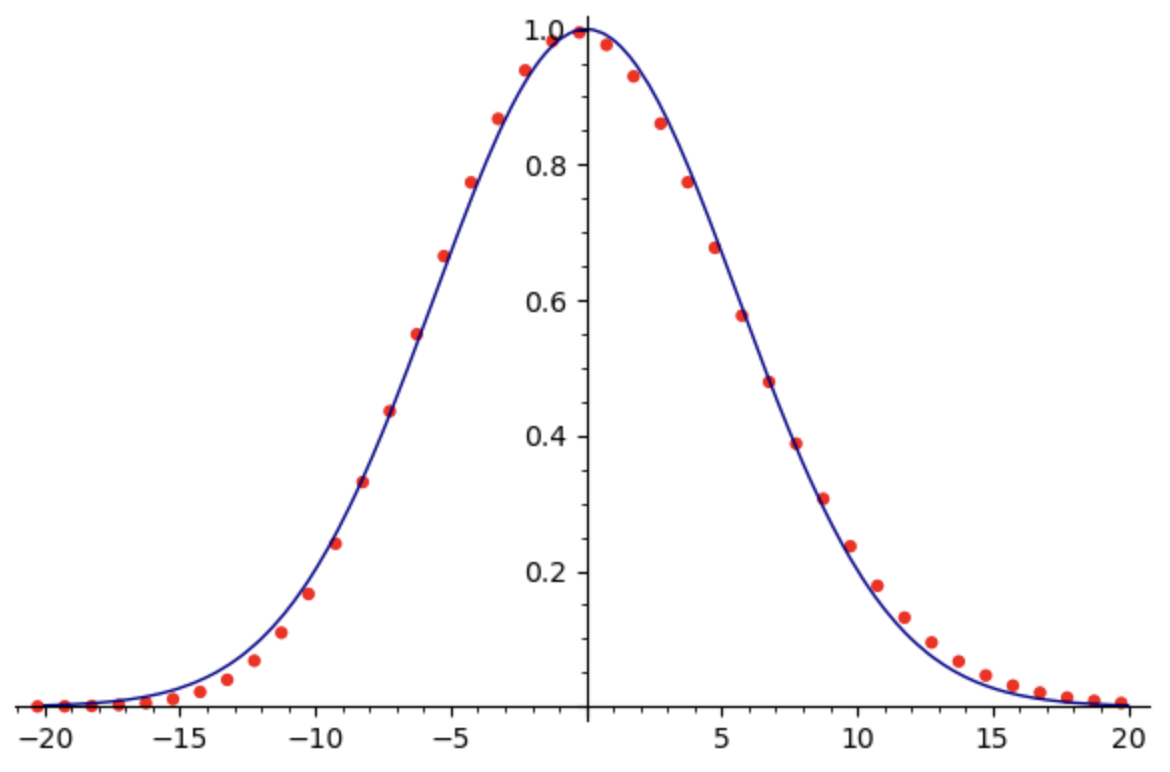}
\includegraphics[scale=0.42]{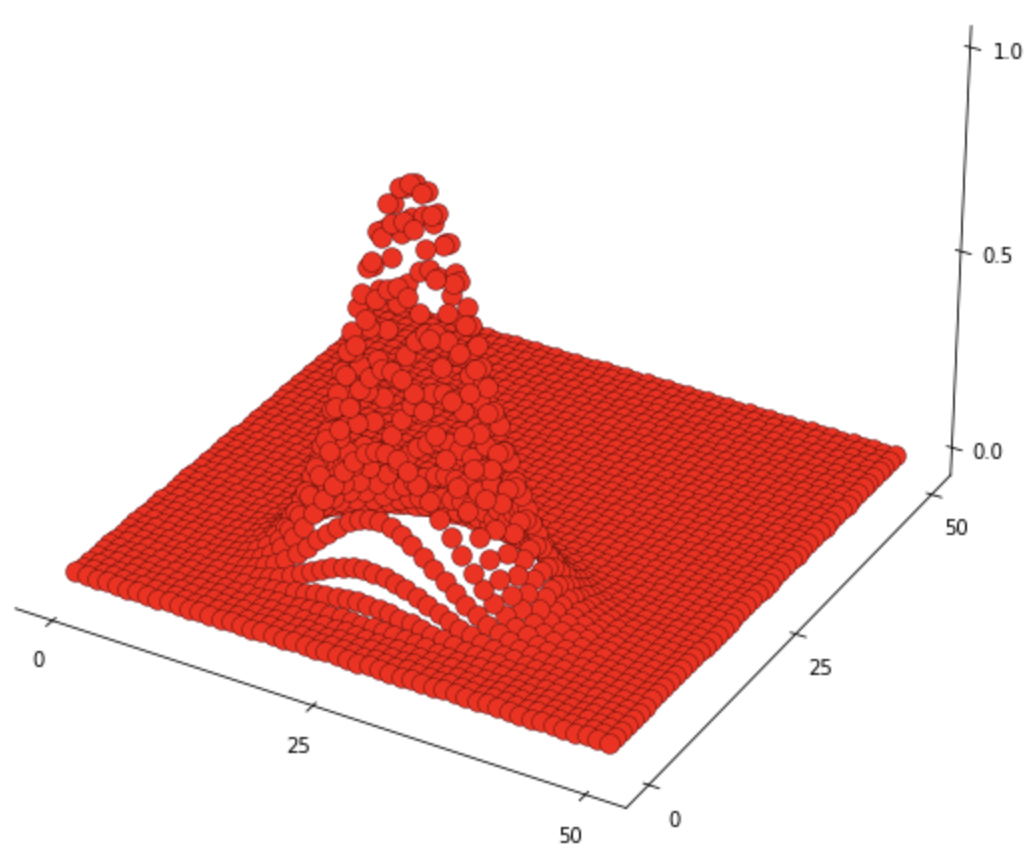}
\caption{\emph{Left:} The coefficients of $[t^{100}]\left(1-z_1t-t^2/(1-t)\right)^{-1}$ divided by the projected maximum coefficient compared to their limiting normal distribution and shifted to put the maximum at the origin. \emph{Right:} The coefficients of $[t^{100}]\left(1-z_1t-z_2t^2-t^3/(1-t)\right)^{-1}$ divided by the projected maximum coefficient and shifted to put the maximum at the origin.}
\label{fig:compplots}
\end{figure}

The flexibility of Theorem~\ref{thm:LCLTgen} allows us to modify the restrictions on the elements appearing or tracked among the compositions under consideration. For instance, for any finite set $\Omega \subset \mathbb{Z}_{>0}$ the multivariate generating function enumerating compositions by size and number of times each element of $\Omega=\{\omega_1,\dots,\omega_d\}$ occurs is
$$F(\bz,t) = \frac{1}{1-\sum_{k=1}^d (z_k-1)t^{\omega_k} - \frac{t}{1-t}},$$
where $z_k$ tracks the number of occurrences of $\omega_k$. The hypotheses of Theorem~\ref{thm:LCLTgen} still hold, so a local central limit theorem applies (although, of course, the parameters of the limiting distribution will depend on which summands are tracked).

Even more generally, we can restrict our compositions to summands in some set $\Lambda \subset \mathbb{Z}_{>0}$ and track occurrences of elements in a finite subset $\Omega=\{\omega_1,\dots,\omega_d\} \subset \Lambda$. The relevant multivariate generating function becomes
\[ F(\bz,t) = \frac{1}{1-\sum_{k=1}^d (z_k-1)t^{\omega_k} - \sum_{k \in \Omega}t^k}. \]
When the elements of $\Omega$ are coprime, Theorem~\ref{thm:LCLTgen} applies, so a local central limit theorem holds (when the elements of $\Omega$ have greatest common divisor larger than $1$ then some intervention is needed to determine which singularities determine asymptotic behaviour).

\begin{rem}
Restricting compositions to use only the numbers $1,\dots,d+1$ and tracking occurrences of those numbers gives the same generating function as the family of restricted permutations discussed above, whose behaviour is described in Theorem~\ref{thm:LCLT}.
\end{rem}

\subsubsection*{\texorpdfstring{$n$}{n}-Colour Compositions with Tracked Summands}
\label{sec:nColComps}

Finally, we consider the class of \emph{$n$-colour compositions}, introduced by Agarwal~\cite{Agarwal2000} and studied in Gibson et al.~\cite{GibsonGrayWang2018}, where each integer $i$ is coloured by one of $i$ available colours (each summand must be coloured and different colourings give distinct $n$-colour compositions). 

\begin{examplex}
The twenty-one $n$-colour compositions of 4 are
\begin{align*}
& 4_1, \text{ } 4_2, \text{ } 4_3, \text{ } 4_4,\\
& 3_11_1, \text{ } 3_21_1, \text{ } 3_31_1, \text{ } 1_13_1, \text{ } 1_13_2, \text{ } 1_13_3,\\
& 2_12_1, \text{ } 2_12_2, \text{ } 2_22_1, \text{ } 2_22_2,\\
& 2_11_11_1, \text{ } 2_21_11_1, \text{ } 1_12_11_1, \text{ } 1_12_21_1, \text{ } 1_11_12_1, \text{ } 1_11_12_2,\\
& 1_11_11_11_1,
\end{align*}
where the subscripts denote the colours assigned to each summand.
\end{examplex}

Once again, we fix a positive integer $d$ and track the number of times each element of $\{1,\dots,d\}$ occurs. If $\mS_n$ is the class of positive integers where each integer $i$ is coloured with one of $i$ possible colours then
\begin{align*}
  \mS_n(\bz, t) &= z_1t+2z_2t^2 + \cdots + dz_dt^d + (d+1)t^{d+1} + (d+2)t^{d+2} + \cdots \\
  &= z_1t+2z_2t^2 + \cdots + dz_dt^d + t\frac{d}{dt}\left(\frac{t^{d+1}}{1-t}\right)\\
  &= z_1t+2z_2t^2 + \cdots + dz_dt^d + \frac{dt^{d+1}}{1-t} + \frac{t^{d+1}}{(1-t)^2}
\end{align*}
is the multivariate generating function of positive integers where $z_k$ tracks the number of occurrences of $k$, so the corresponding multivariate generating function for $n$-colour compositions is
$$C_n(\bz, t) = \frac{1}{1-z_1t-2z_2t^2 - \cdots - dz_dt^d - \frac{dt^{d+1}}{1-t} - \frac{t^{d+1}}{(1-t)^2}}.$$
As expected, the hypotheses of Theorem~\ref{thm:LCLTgen} hold, and we get a (messier) LCLT whose parameters are given explicitly in the Sage notebooks corresponding to this paper. Similar to our last example, we may further restrict which elements (or colours!) are allowed or tracked, and immediately get LCLTs. The proofs for such extensions follow similarly to those for integer compositions.

\section{Proof of Theorem~\ref{thm:LCLTgen}}
\label{sec:GenForm}

We now prove Theorem~\ref{thm:LCLTgen} by applying the outline in Figure~\ref{fig:schema2} to verify the conditions of Proposition~\ref{prop:LCLT2}. To that end, assume the hypotheses of Theorem~\ref{thm:LCLTgen} hold.

\subsection*{Step 1: Finding the correct direction.}
Following the outline of Figure~\ref{fig:schema2}, we substitute $\bw = (\bone, \rho)$ into the smooth critical point equations~\eqref{eq:crit} to find that $\bw$ is a critical point in the direction $(\bm,1)$ where
$$\bm = \left(\frac{-q_1(\rho)}{\rho P'(\rho)},\dots,\frac{-q_d(\rho)}{\rho P'(\rho)}\right).$$

\subsection*{Steps 2 and 3: Establishing minimality.}
Suppose $H(s, \ldots, s, t) = 0$ for $0<s<1$ and $0<t<\rho$. Then $s\sum_{k=1}^dq_k(t) = 1-q(t)$, and we cannot have $\sum_{k=1}^dq_k(t)=0$ since this implies $q(0)=1$ which would violate the fact that $q$ has no constant, so 
$$s = \frac{1-q(t)}{\sum_{k=1}^dq_k(t)}.$$ 
Our assumptions imply that the polynomial $\sum_{k=1}^dq_k(t)$ and series expansion of $q(t)$ have non-negative coefficients, and $0<q(\rho)<1$, which implies that $|s|$ increases as $0<t<\rho$ decreases. Thus $H(s,\ldots,s,s\rho) \neq 0$ for all $0 < s < 1$, and $\bw$ is minimal by applying Lemma~\ref{lemma:minimal} to the function $1/H(\bz,t) = 1/(1 - S(\bz, t))$ whose series expansion at the origin has non-negative coefficients. Our assumptions imply that $S(\bz, t)$ is an aperiodic power series with non-negative coefficients, so Proposition~\ref{prop:strictmin} implies that no other singularities (including no other critical points) have the same coordinate-wise modulus as $\bw$. 

\subsection*{Step 4a: Computing an $LU$-factorization.}

We now prove that the phase Hessian matrix defined by~\eqref{eq:Hess}, which simplifies to~\eqref{eq:Hgen} in our case, has non-zero determinant. To compute this symbolic determinant, we originally used the Sage computer algebra system to compute and factor the Hessian determinant for the permutation generating function described by Proposition~\ref{prop:ratGF} in small dimensions. Observing a pattern in the factors, we were able to conjecture, and then \emph{a posteriori} prove, an $LU$-factorization for the Hessian matrix in this case, and then extend to the general case. This approach immediately gives the Hessian determinant: if $\mH U = L$ for a lower triangular matrix $L$ and an upper triangular matrix $U$ with diagonal elements equal to 1 then $\det \mH$ is simply the product of the diagonal elements of $L$. 

\begin{rem}
When the dimension $d$ is fixed, the matrix equation $\mH U = L$ defines an explicit system of equations in terms of the entries of $U$ and $L$, so it is often possible to computationally determine suitable $U$ and $L$ in low dimension, conjecture their general form, then prove this inference. This approach to symbolic determinants is described in the well-known treatise of Krattenthaler~\cite{Krattenthaler1999}, which attributes the popularization of such a `guess-and-check' $LU$-factorization to George Andrews after he used it to great effect in a variety of papers starting in the 1970s.
\end{rem}

A companion Sage notebook to this paper gives a procedure to calculate $U$ and $L$ for any fixed $d$. Studying the numerator and denominator of the rational function entries we are able to deduce certain patterns, such as the denominators being constant down columns, which leads us to conjecture that in general dimension $\mH U = L$ where 
\\
\begin{minipage}{0.4\textwidth}
\[
  \text{U}_{ij} = \begin{cases}
   \frac{q_j(\rho)g_{ij}}{r_{j}} & i < j\\[+1mm]
   1 & i = j \\[+1mm]
   0 & i > j
  \end{cases}
\]
\end{minipage}
\begin{minipage}{0.1\textwidth}
and
\end{minipage}
\begin{minipage}{0.4\textwidth}
\[  L_{ij} = \begin{cases}
    0 & i < j\\[+1mm]
    \frac{-q_j(\rho)r_{j+1}}{P'(\rho)\rho r_{j}} & i = j\\[+1mm]
    \frac{q_j(\rho)q_i(\rho)s_{ij}}{P'(\rho)\rho r_j} & i > j
    \end{cases}
\]
\end{minipage}
\begin{minipage}{0.05\textwidth}
\vspace{10mm}
\end{minipage}

\noindent for
\begin{align*}
r_{j} &= P'(\rho)^2\rho-P''(\rho)\rho A_j + 2P'(\rho)\rho B_j - P'(\rho) A_j - \rho \left(A_jD_j - B_j^2 \right) 
\\[+2mm]
g_{ij} &= P'(\rho) + P''(\rho)\rho
-\rho\left(\frac{q'_i(\rho)}{q_i(\rho)}+\frac{q'_j(\rho)}{q_j(\rho)}\right)\left(P'(\rho) + B_j-q'_i(\rho)\right)
+\rho\left(D_j-\frac{q'_i(\rho)^2}{q_i(\rho)}\right)
+\frac{\rho q'_i(\rho)q'_j(\rho)}{q_i(\rho)q_j(\rho)}\left(A_j-q_i(\rho)\right)
\\[+2mm]
s_{ij} &= P''(\rho)\rho 
+ \rho D_j
-\rho\left(\frac{q'_i(\rho)}{q_i(\rho)}+\frac{q'_j(\rho)}{q_j(\rho)}\right)\left(P'(\rho) + B_j\right)
+P'(\rho)
+ \left(\frac{\rho q'_i(\rho)q'_j(\rho)}{q_i(\rho)q_j(\rho)}\right)A_j
\end{align*}
with
$$A_j = \sum_{k=1}^{j-1}q_k(\rho), \qquad 
B_j = \sum_{k=1}^{j-1}q'_k(\rho), \qquad
D_j = \sum_{k=1}^{j-1}\frac{q_k'(\rho)^2}{q_k(\rho)}.
$$ 
Although these formulas are quite involved, we note that by keeping $P'(\rho)$ and $P''(\rho)$ as symbolic parameters the entries of $U$ and $L$ are independent of $d$, which allows us to algorithmically verify that $\mH U = L$ with the aid of a computer algebra system. We break this verification into three cases depending on the behaviour of the entries of $U$ and $L$: see Figure~\ref{fig:GenCases} for an illustration of the different cases on a $4 \times 4$ matrix. 

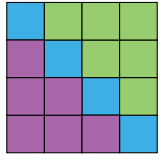
\begin{figure}[t]
\centering
\begin{tikzpicture}
    draw[step=0.5cm, gray, very thin] (0,0) grid (2,2);
    \filldraw[CornflowerBlue, draw = black] (0,1.5) rectangle (0.5, 2);
    \filldraw[Plum!70!white, draw = black] (0,0) rectangle (0.5, 0.5);
    \filldraw[Plum!70!white, draw = black] (0,0.5) rectangle (0.5, 1);
    \filldraw[Plum!70!white, draw = black] (0,1) rectangle (0.5, 1.5);
    \filldraw[CornflowerBlue, draw = black] (0.5,1) rectangle (1, 1.5);
    \filldraw[CornflowerBlue, draw = black] (1,0.5) rectangle (1.5, 1);
    \filldraw[CornflowerBlue, draw = black] (1.5,0) rectangle (2, 0.5);
    \filldraw[YellowGreen, draw = black] (0.5,1.5) rectangle (1, 2);
    \filldraw[YellowGreen, draw = black] (1,1) rectangle (1.5, 1.5);
    \filldraw[YellowGreen, draw = black] (1,1.5) rectangle (1.5, 2);
    \filldraw[YellowGreen, draw = black] (1.5,0.5) rectangle (2, 1);
    \filldraw[YellowGreen, draw = black] (1.5,1) rectangle (2, 1.5);
    \filldraw[YellowGreen, draw = black] (1.5,1.5) rectangle (2, 2);
    \filldraw[Plum!70!white, draw = black] (0.5,0) rectangle (1, 0.5);
    \filldraw[Plum!70!white, draw = black] (0.5,0.5) rectangle (1, 1);
    \filldraw[Plum!70!white, draw = black] (1,0) rectangle (1.5, 0.5);
\end{tikzpicture}
\caption{The three cases, distinguished by colour, for $d=4$.}
\label{fig:GenCases}
\end{figure}

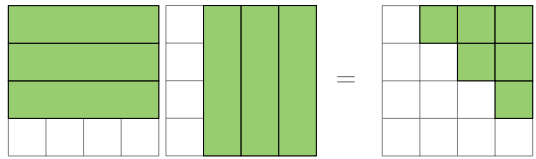
\begin{figure}[h!]
\centering
\begin{minipage}{0.12\textwidth}
  \begin{tikzpicture}
    \draw[step=0.5cm, gray, very thin] (0,0) grid (2,2);
    \filldraw[YellowGreen, draw = black] (0,0.5) rectangle (2, 1);
    \filldraw[YellowGreen, draw = black] (0,1) rectangle (2, 1.5);
    \filldraw[YellowGreen, draw = black] (0,1.5) rectangle (2, 2);
  \end{tikzpicture}
\end{minipage}
\begin{minipage}{0.13\textwidth}
  \begin{tikzpicture}
    \draw[step=0.5cm, gray, very thin] (0,0) grid (2,2);
    \filldraw[YellowGreen, draw = black] (0.5,0) rectangle (1, 2);
    \filldraw[YellowGreen, draw = black] (1,0) rectangle (1.5, 2);
    \filldraw[YellowGreen, draw = black] (1.5,0) rectangle (2, 2);
  \end{tikzpicture}
\end{minipage}
\begin{minipage}{0.03\textwidth}
=
\end{minipage}
\begin{minipage}{0.1\textwidth}
  \begin{tikzpicture}
    \draw[step=0.5cm, gray, very thin] (0,0) grid (2,2);
    \filldraw[YellowGreen, draw = black] (0.5,1.5) rectangle (1, 2);
    \filldraw[YellowGreen, draw = black] (1,1) rectangle (1.5, 1.5);
    \filldraw[YellowGreen, draw = black] (1,1.5) rectangle (1.5, 2);
    \filldraw[YellowGreen, draw = black] (1.5,0.5) rectangle (2, 1);
    \filldraw[YellowGreen, draw = black] (1.5,1) rectangle (2, 1.5);
    \filldraw[YellowGreen, draw = black] (1.5,1.5) rectangle (2, 2);
  \end{tikzpicture}
\end{minipage}
\caption*{Case 1: $j > i$}
\end{figure}

\paragraph{Case 1 ($j > i$):}
As $U$ is upper-triangular, we have
\[ (\mH U)_{ij} = \sum_{a=1}^j \mH_{ia}U_{aj} = \sum_{\substack{1 \leq a < j \\ a \neq i}} \mH_{ia} U_{aj} + \mH_{ii}U_{ij} + \mH_{ij} \]
where we split the sum in such a way that the summands in the indefinite series have a uniform definition. Using Sage for algebraic manipulations, we see that
\[
  \mH_{ia}U_{aj} = \frac{\alpha_{ij}\left(\frac{q'_a(\rho)^2}{q_a(\rho)}\right) + \beta_{ij} q_a(\rho) + \gamma_{ij} q'_a(\rho)}{\rho^2P'(\rho)^3r_j}
\]
where 
\begin{align}
\alpha_{ij} &= \left(P'(\rho)q_j(\rho) - q'_j(\rho)A_j + q_j(\rho)B_j\right)P'(\rho)q_i(\rho)\rho^2 \label{eq:alpha} \\[+2mm] 
\beta_{ij} &= {\Big[}
P'(\rho)\left(q'_j(\rho)\rho-q_j(\rho)\right) - P''(\rho)q_j(\rho)\rho + q'_j(\rho)\rho B_j - q_j(\rho)\rho D_j {\Big]} \label{eq:beta}\\
& \hspace{0.2in} \cdot {\Big[}(P'(\rho)\left(q'_i(\rho)\rho - q_i(\rho)\right) - P''(\rho)q_i(\rho)\rho {\Big]} \notag
\end{align}
and
\begin{align}
\begin{split}
\gamma_{ij} =& 
{\Big[}q'_j(\rho)q_i(\rho) + q'_i(\rho)q_j(\rho) {\Big]} \rho^2 P'(\rho)^2
- 2 {\Big[} P'(\rho) + P''(\rho)\rho {\Big]}\rho P'(\rho) q_i(\rho)q_j(\rho) \\[+2mm]
& + {\Big[} P''(\rho)q'_j(\rho)q_i(\rho)\rho - P'(\rho)q'_i(\rho)q'_j(\rho)\rho + P'(\rho)q'_j(\rho)q_i(\rho) {\Big]}\rho A_j  \\[+2mm]
& + {\Big[} q'_j(\rho)q_i(\rho)\rho + q'_i(\rho)q_j(\rho)\rho - q_i(\rho)q_j(\rho) {\Big]} \rho P'(\rho) B_j - P''(\rho)q_i(\rho)q_j(\rho)\rho^2 B_j  \\[+2mm]
& - P'(\rho)q_i(\rho)q_j(\rho)\rho^2D_j.
\end{split} \label{eq:gamma}
\end{align}

Since $\alpha_{ij}$, $\beta_{ij}$, $\gamma_{ij},$ and the denominator in our above expression for $\mH_{ia}U_{aj}$ are not dependent on $a$,

\begin{align*}
  \sum_{\substack{1 \leq a < j \\ a \neq i}} \mH_{ia} U_{aj} &= \sum_{\substack{1 \leq a < j \\ a \neq i}}\frac{\alpha_{ij}\left(\frac{q'_a(\rho)^2}{q_a(\rho)}\right) + \beta_{ij} q_a(\rho) + \gamma_{ij} q'_a(\rho)}{\rho^2P'(\rho)^3r_j} \\
  &= \frac{1}{\rho^2P'(\rho)^3r_j} \left[\alpha_{ij} \sum_{\substack{1 \leq a < j \\ a \neq i}}\frac{q'_a(\rho)^2}{q_a(\rho)} + \beta_{ij} \sum_{\substack{1 \leq a < j \\ a \neq i}}q_a(\rho) + \gamma_{ij} \sum_{\substack{1 \leq a < j \\ a \neq i}}q'_a(\rho)\right] \\[+2mm]
  &= \frac{1}{\rho^2P'(\rho)^3r_j} \left[\alpha_{ij} \left(D_j - \frac{q'_i(\rho)^2}{q_i(\rho)}\right)+ \beta_{ij} \left(A_j - q_i(\rho)\right) + \gamma_{ij} \left(B_j - q'_i(\rho)\right)\right].
\end{align*}
Substitution of our above formulas then gives, after some algebraic simplification by Sage, that 
$$
\sum_{\substack{1 \leq a < j \\ a \neq i}} \mH_{ia} U_{aj} + \mH_{ii}U_{ij} + \mH_{ij}=0
$$ 
in this case, as required.

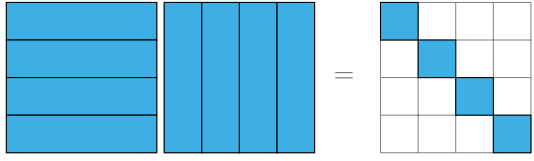
\begin{figure}[h!]
\centering
\begin{minipage}{0.12\textwidth}
  \begin{tikzpicture}
    \draw[step=0.5cm, gray, very thin] (0,0) grid (2,2);
    \filldraw[CornflowerBlue, draw = black] (0,1.5) rectangle (2, 2);
    \filldraw[CornflowerBlue, draw = black] (0,1) rectangle (2, 1.5);
    \filldraw[CornflowerBlue, draw = black] (0,0.5) rectangle (2, 1);
    \filldraw[CornflowerBlue, draw = black] (0,0) rectangle (2, 0.5);
  \end{tikzpicture}
\end{minipage}
\begin{minipage}{0.13\textwidth}
  \begin{tikzpicture}
    \draw[step=0.5cm, gray, very thin] (0,0) grid (2,2);
    \filldraw[CornflowerBlue, draw = black] (0,0) rectangle (0.5, 2);
    \filldraw[CornflowerBlue, draw = black] (0.5,0) rectangle (1, 2);
    \filldraw[CornflowerBlue, draw = black] (1,0) rectangle (1.5, 2);
    \filldraw[CornflowerBlue, draw = black] (1.5,0) rectangle (2, 2);
  \end{tikzpicture}
\end{minipage}
\begin{minipage}{0.03\textwidth}
=
\end{minipage}
\begin{minipage}{0.1\textwidth}
  \begin{tikzpicture}
    \draw[step=0.5cm, gray, very thin] (0,0) grid (2,2);
    \filldraw[CornflowerBlue, draw = black] (0,1.5) rectangle (0.5, 2);
    \filldraw[CornflowerBlue, draw = black] (0.5,1) rectangle (1, 1.5);
    \filldraw[CornflowerBlue, draw = black] (1,0.5) rectangle (1.5, 1);
    \filldraw[CornflowerBlue, draw = black] (1.5,0) rectangle (2, 0.5);
  \end{tikzpicture}
\end{minipage}
\caption*{Case 2: $i = j$}
\end{figure}

\paragraph{Case 2 ($i = j$):}
Since $U$ is upper-triangular we again have
\[ (\mH U)_{jj} = \sum_{a=1}^d \mH_{ja}U_{aj} = \sum_{a=1}^j \mH_{ja}U_{aj} = \sum_{a=1}^{j-1} \mH_{ja}U_{aj} + \mH_{jj},\] 
where we separate the case $a=j$ since the entries of $U$ have a different definition on the diagonal and above the diagonal. It is therefore sufficient to prove 
\[ \sum_{a=1}^{j-1} \mH_{ja}U_{aj} + \mH_{jj} = L_{jj} \]
using the definitions above. Analogously to Case 1,
\[
  \mH_{ja}U_{aj} = \frac{\alpha_{jj}\left(\frac{q'_a(\rho)^2}{q_a(\rho)}\right) + \beta_{jj} q_a(\rho) + \gamma_{jj} q'_a(\rho)}{\rho^2P'(\rho)^3r_j}
\]
where $\alpha_{jj}$, $\beta_{jj}$ and $\gamma_{jj}$ are defined in~\eqref{eq:alpha},~\eqref{eq:beta} and~\eqref{eq:gamma} respectively. Since $\alpha_{jj}$, $\beta_{jj}$, $\gamma_{jj},$ and the denominator in this expression for $\mH_{ja}U_{aj}$ are not dependent on $a$, we can algebraically simplify to get
\begin{align*}
  \sum_{a=1}^{j-1} \mH_{ja}U_{aj} &= \sum_{a=1}^{j-1}\left[\frac{\alpha_{jj}\left(\frac{q'_a(\rho)^2}{q_a(\rho)}\right) + \beta_{jj} q_a(\rho) + \gamma_{jj} q'_a(\rho)}{\rho^2P'(\rho)^3r_j}\right] \\[+2mm]
  &= \frac{\alpha_{jj} D_j + \beta_{jj} A_j + \gamma_{jj} B_j}{\rho^2P'(\rho)^3r_j}.
\end{align*}
Substitution of our above formulas then gives, after some algebraic simplification by Sage, that 
$$
\sum_{a=1}^{j-1} \mH_{ja}U_{aj} + \mH_{jj} - L_{jj} = 0
$$ 
in this case, as desired.

\paragraph{Case 3 ($j < i$):}
At this point, we have proven that $\mH U$ is lower-triangular and described its diagonal entries, which gives the desired determinant. Although not needed to establish Theorem~\ref{thm:LCLTgen}, for completeness we prove the claimed expression above for the below diagonal entries in our companion Sage notebook.

\subsection*{Step 3b: Computing the Hessian determinant}

Our $LU$-factorization expresses the Hessian determinant as the product of the diagonal entries of the lower-triangular matrix $L$. In fact, the form of these diagonal entries causes cancellation of many terms, leading to a relatively compact final answer,
\begin{align*}
\det\mH = 
 \prod^{d}_{m = 1} L_{mm} = \prod^{d}_{m=1} \frac{-q_m(\rho)r_{m+1}}{P'(\rho)\rho r_m}
 &= \frac{(-1)^d\left(\prod_{k=1}^{d}q_k(\rho)\right)r_{d+1}}{P'(\rho)^{d}\rho^{d}r_1}\\[+2mm]
&= \frac{(-1)^d\left(\prod_{k=1}^{d}q_k(\rho)\right)r_{d+1}}{P'(\rho)^{d+2}\rho^{d+1}}.
\end{align*}
Furthermore, 
\[
P(\rho) = 1 - q(\rho) - \sum_{k=1}^dq_k(\rho) = 0 \quad\text{and}\quad
P'(\rho) = - q'(\rho) - \sum_{k=1}^dq'_k(\rho)
\]
so that $A_{d+1} = \sum_{k=1}^dq_k(\rho) = 1 - q(\rho)$ and $B_{d+1} = \sum_{k=1}^dq'_k(\rho) = -q'(\rho) - P'(\rho)$. Making these substitutions into our definition of $r_{d+1}$ gives, after some algebraic simplification, that the Hessian determinant has the form~\eqref{eq:Hgendet}. 

\subsection*{Step 3c: Non-singularity of the determinant}
It remains to show that the expression in~\eqref{eq:Hgendet} is non-zero under our assumptions. Each $q_k(\rho) \neq 0$ because $\rho>0$ and the $q_k$ have non-negative coefficients, so it is sufficient to prove that
$$(q(\rho)-1)\left(\rho P''(\rho) + P'(\rho)+\rho\sum_{k=1}^{d}\frac{q'_k(\rho)^2}{q_k(\rho)}\right) + q'(\rho)^2\rho > 0.$$
First, we note that $q'(\rho)^2\rho > 0$ (as $q(t)$ is non-constant with non-negative coefficients and $\rho > 0$) and $q(\rho) - 1 \leq 0$, so it is enough to prove that
$$\rho P''(\rho) + P'(\rho) + \rho\sum_{k=1}^d\frac{q'_k(\rho)^2}{q_k(\rho)} \leq 0.$$ 
Because 
$$ \rho P''(\rho) + P'(\rho) + \rho\sum_{k=1}^d\frac{q'_k(\rho)^2}{q_k(\rho)} = \rho {\Big[}-q''(\rho) - \sum_{k=1}^dq_k''(\rho){\Big]} -q_k'(\rho) - \sum_{k=1}^dq_k'(\rho) + \rho\sum_{k=1}^d\frac{q'_k(\rho)^2}{q_k(\rho)}$$
and $q'(\rho)$ and $\rho q''(\rho)$ are non-negative, this holds if
$$ \rho\sum_{k=1}^d\frac{q'_k(\rho)^2}{q_k(\rho)} \leq \rho \sum_{k=1}^dq_k''(\rho) + \sum_{k=1}^dq_k'(\rho). $$
Now, if $f(z) = a_1z + a_2z^2 + \cdots + a_sz^s$ is any non-zero polynomial vanishing at the origin with non-negative coefficients then
\begin{align*}
zf'(z)^2 &=z\left(a_1 + 2a_2 z + \cdots + s a_{s} z^{s-1}\right)^2 \\[+2mm]
&\leq z\left(a_1 + a_2 z + \cdots + a_{s} z^{s-1}\right)\left(a_1 + 2^2 a_2 z + \cdots + s^2 a_{s} z^{s-1}\right) \\[+2mm]
& = f(z) \left(zf'(z)\right)' \\[+2mm]
& = f(z) \left(zf''(z) + f'(z)\right)
\end{align*}
for any $z>0$, where the first inequality holds because each term in the expansion is non-negative and $2ij \leq i^2 + j^2$ for all $i,j \in \mathbb{N}$. Thus,
$$
\frac{\rho q_k'(\rho)^2}{q_k(\rho)} \leq \rho q_k''(\rho) + q_k'(\rho)
$$
for all $k$, and summing over $k$ gives the desired inequality. The Hessian determinant is therefore non-zero.

\begin{rem}
When each of the $q_k$ are monomials this inequality becomes an equality.
\end{rem}

\subsection*{Step 4: Checking final hypotheses}

To conclude Theorem~\ref{thm:LCLTgen} from Proposition~\ref{prop:LCLT} it remains only to note that $G(\bone, \rho) \neq 0$, which is one of our assumptions, and that $H_t(\bw) = P'(\rho) \neq 0$ because none of the $q_k(t)$ and $q(t)$ have constant terms and we assume the denominator is not constant.

\section{Conclusion}
\label{sec:Conclusion}

The methods of analytic combinatorics in several variables, while perhaps daunting to some outside users due to its reliance on a wide breadth of mathematical techniques, provide some of the most powerful tools for the study of multivariate generating functions. Although Bender and Richmond~\cite{BenderRichmond1983} already provided techniques to prove local central limit theorems for a variety of combinatorial generating functions, verifying required conditions on analytic regions for the generating functions is too expensive to implement in a general, practical, algorithm. Using results of ACSV, a better understanding of the singular sets of multivariate generating functions yields such an algorithm, in addition to providing a framework for further generalizations. It is our hope that by putting the results of ACSV into context with past probabilistic work, giving a simple outline of how to apply the results, implementing the results in a computer algebra package, and using the results to prove a family of limit theorems, that readers are inspired to look further into this growing area of combinatorics.

\section*{Acknowledgements}
The authors thank Persi Diaconis for telling them about the interesting permutations with restricted cycles limit theorem, an anonymous referee for their insightful comments that helped to improve the structure and content of the paper, and Andrew Luo for associated work in Sage. SM was funded by NSERC Discovery Grant RGPIN-2021-02382 and TR was funded by an NSERC USRA and an NSERC CGS M.

\bibliographystyle{plain}
\bibliography{ref}

\end{document}